\documentclass[11pt,a4paper,reqno]{amsart}
\usepackage{amsmath,amssymb,graphics,epsfig,color,enumerate}
\usepackage{dsfont}  
\usepackage{verbatim}
\definecolor{refkey}{gray}{.75}
\definecolor{labelkey}{gray}{.5}
 \definecolor{darkgreen}{rgb}{0,0.4,0}

\definecolor{light}{gray}{0.9}

\newtheorem{theorem}{Theorem}[section]
\newtheorem{proposition}[theorem]{Proposition}
\newtheorem{lemma}[theorem]{Lemma}
\newtheorem{corollary}[theorem]{Corollary}
\newtheorem{claim}[theorem]{Claim}
\newtheorem{definition}[theorem]{Definition}

\newtheorem{remark}[theorem]{Remark}

\numberwithin{equation}{section}
\numberwithin{theorem}{section}
\newtheorem{theoremA}{Theorem}
\newcommand{\mc}[1]{{\mathcal #1}}
\newcommand{\bb}[1]{{\mathbb #1}}

\renewcommand{\epsilon}{\varepsilon}
\renewcommand{\tilde}{\widetilde}
\renewcommand{\hat}{\widehat}

\renewcommand{\div}{\mathop{\rm div}\nolimits}
\newcommand{\Ent}{\mathop{\rm Ent}\nolimits}

\newcommand{\sce}{\mathop{\rm sc^-\!}\nolimits}

\definecolor{light}{gray}{.9}

\newcommand{\cA}{\ensuremath{\mathcal A}}
\newcommand{\cB}{\ensuremath{\mathcal B}}
\newcommand{\cC}{\ensuremath{\mathcal C}}
\newcommand{\cD}{\ensuremath{\mathcal D}}

\newcommand{\cG}{\ensuremath{\mathcal G}}

\newcommand{\cI}{\ensuremath{\mathcal I}}
\newcommand{\cJ}{\ensuremath{\mathcal J}}
\newcommand{\cK}{\ensuremath{\mathcal K}}

\newcommand{\cM}{\ensuremath{\mathcal M}}

\newcommand{\cO}{\ensuremath{\mathcal O}}
\newcommand{\cP}{\ensuremath{\mathcal P}}

\newcommand{\cS}{\ensuremath{\mathcal S}}
\newcommand{\cT}{\ensuremath{\mathcal T}}
\newcommand{\cU}{\ensuremath{\mathcal U}}

\newcommand{\bbE}{{\ensuremath{\mathbb E}} }

\newcommand{\bbP}{{\ensuremath{\mathbb P}} }

\newcommand{\bbR}{{\ensuremath{\mathbb R}} }

\newcommand{\bbX}{{\ensuremath{\mathbb X}} }

\newcommand{\bbZ}{{\ensuremath{\mathbb Z}} }

\let\a=\alpha    \let\d=\delta  \let\e=\varepsilon
 \let\g=\gamma       \let\l=\lambda
            
  \let\s=\sigma \let\t=\tau   
  \let\z=\zeta
     \let\L=\Lambda

\newcommand{\rosso}{\textcolor{black}}
\newcommand{\blu}{\textcolor{black}}

\newcommand{\re}{\mathbb{R}_+^E}

\title[LDP's  for  periodic Markov chains]{Level 2.5 large deviations  for  continuous time  Markov chains with time periodic rates}

\author[L.\ Bertini]{L. Bertini}
\address{Lorenzo Bertini \hfill\break \indent
   Dipartimento di Matematica, Universit\`a di Roma `La Sapienza'
   \hfill\break \indent
   P.le Aldo Moro 2, 00185 Roma, Italy}
 \email{bertini@mat.uniroma1.it}

\author[R.\ Chetrite]{R. Chetrite}
\address{Raphael Chetrite \hfill\break \indent
  CNRS, Laboratoire J.A. Dieudonn\'e, Universit\'e C{\^o}te d'Azur
   \hfill\break \indent
   06108 Nice Cedex 2, France. \hfill\break \indent
 Department of Physics, Graduate School of Science
 \hfill\break \indent
Kyoto University, Kyoto 606-8502, Japan
}
 \email{rchetrit@unice.fr}

\author[A.\ Faggionato]{A. Faggionato}
\address{Alessandra Faggionato \hfill\break \indent
  Dipartimento di Matematica, Universit\`a di Roma `La Sapienza'
  \hfill\break \indent
  P.le Aldo Moro 2, 00185 Roma, Italy}
\email{faggiona@mat.uniroma1.it}

\author[D.\ Gabrielli]{D. Gabrielli}
\address{Davide Gabrielli \hfill\break \indent
  DISIM, Universit\`a dell'Aquila
  \hfill\break\indent
  Via Vetoio,  Loc. Coppito, 67100 L'Aquila, Italy}
\email{gabriell@univaq.it}

\begin{document}

\begin{abstract}
We consider an irreducible continuous time Markov chain on a finite
state space and with time periodic jump rates and prove the joint large
deviation principle for the empirical measure and flow and the joint large deviation principle for the empirical measure and current. By contraction we get the large deviation principle of three types of  entropy production flow. We derive some Gallavotti-Cohen duality relations  and discuss some applications.

\medskip

\noindent {\em Keywords}:
time periodic Markov chain, large deviation principle,
empirical measure, flow  and current,
Gallavotti-Cohen duality relation.
\medskip

\noindent{\em AMS 2010 Subject Classification}:
60F10,  
60J27, 
82C05.  

\thanks{The work of L. Bertini and  A. Faggionato  has been  supported  by  PRIN
  20155PAWZB ``Large Scale Random Structures", while the  work of R. Chetrite has been supported by the French Ministry of Education grant
   ANR-15-CE40-0020-01.
 }

\end{abstract}
\maketitle
\thispagestyle{empty}

\section{Introduction}


\rosso{Periodically driven Markov processes are a common setup for several small systems, such as artificial molecular motors.
  Unlike their biological
counterparts, artificial molecular motors are non autonomous and  work under the  effect of time periodic externally-modulated
  stimuli
such as temperature, laser   light, chemical environment.  Periodic forcing is fundamental also  in relation to micro-sized engines. For example, experimental heat
engines driven by periodic temperature variations have been realized
in \cite{Blicke,Martinez} and experimental molecular pumps with
periodic modulation of hamiltonian have been realized in \cite{Eelkema,Li}.
Due to their  numerous applications inside nanotechnologies, in the last years
 periodically driven Markov processes have received  much attention in the  thermodynamic theory of small systems or, equivalently, stochastic thermodynamics \cite{Sek,Se}. Several theoretical results have been obtained for what concerns
 linear response and Onsager reciprocity relations \cite{Brandner,Joubaud,Proesmans1,Proesmans2,Ray,Yzumida1,Yzumida3}, and    no-go theorems in  stochastic pumping \cite{Chernyak2,Joubaud,Maes2,Rahav,Rotskoff}.
Time periodic forcing is also at the basis  of stochastic resonance phenomena \cite{GHJM}, i.e of  the amplification of a weak periodic signal by means of noise. }

\rosso{We  consider here  an irreducible continuous time Markov chain on a finite
state space $V$ with time periodic jump rates, having  period $T_0$.  We focus on  large deviations (shortly, LD) at large times of the empirical measure, flow  and current.
Roughly, referring  to a time window $[0,T]$, for each  $y\in V$ the empirical measure $\bar \mu_T(y)$   counts the fraction of time spent by the system  in the state  $y$. For each pair of states $y,z \in V$ the empirical flow $\bar Q_T(y,z)$   counts the number of jumps  from $y$ to $z$  per unit time, while the empirical current $\bar J_T (y,z)$ is given by the difference  $\bar Q_T (y,z)-\bar Q_T (z,y)$. The above objects enter naturally in several applications. Considering for example the dynamics of a  molecular motor for which an ATP hydrolysis takes place simultaneously to a transitions from  state $y$ to state $z$,  $ T Q_T(y,z)$ gives the number of hydrolyzed ATP's in the time window $[0,T]$.
 To have explicit LD functionals, we consider also  extended versions $\mu^{(n)} , Q^{(n)}, J^{(n)}$ of the above empirical measure, flow and current, keeping record of  time $t$ apart integer multiples of the period $T_0$. More precisely,
 $\mu^{(n)}(y,dt)$ is defined as  the time per  period spent   at state $y$  during  the $n$ infinitesimal time intervals  $[t,t+dt)$, $[t+T_0, t+T_0+dt)$,.., $[t+( n-1)T_0, t +(n-1) T_0+dt)$.
$Q^{(n)}(y,z,dt)$   equals  the number, per period,  of jumps   from $y$ to $z$ performed in the above $n$  infinitesimal time intervals, while $J^{(n)}(y,z,dt)$ is defined as the difference $Q^{(n)}(y,z,dt)-Q^{(n)}(z,y,dt)$.}

\rosso{Although our initial objects of investigation are  given by the empirical measure, flow and current, to get control on large deviations it is crucial to include more information and deal with their extended versions. Indeed,
the technical  core of our analysis is the   derivation of
  the LD principle   for the joint extended empirical measure and flow
   $\bigl(\mu^{(n)}, Q^{(n)}\bigr)$
   as $n$ goes to $\infty$ (cf. Theorem \ref{teo2}).
   Roughly, we get that
  \[ \bbP\left(  (\mu^{(n)}, Q^{(n)})\approx (\mu,Q) \right) \asymp e^{- n I(\mu,Q)}\,, \qquad n\gg 1\,,
  \]
  for a suitable  explicit rate functional $I(\cdot ,\cdot)$.
       By contraction, we obtain the  LD principle  for the joint extended empirical measure and current  $\bigl(\mu^{(n)}, J^{(n)}\bigr)$ (cf. Theorem \ref{teo3}). These LD principles  correspond to  level 2.5 (see below) and remarkably admit  explicit LD rate functionals. By contraction, we derive LD principles for the empirical  measure, flow and current  $\bar \mu_T$, $\bar Q_T$, $\bar J_T$ (cf. Theorem \ref{teo1} , Remark \ref{parto} and Remark \ref{nuovo}). As an application of the above results, after introducing  several forms of entropy production, we obtain the associated LD  principles by contraction. Moreover, we derive some Gallavotti-Cohen duality relations at the level 2.5 and show that, by  projection, theses relations  imply other Gallavotti-Cohen duality relations  for the entropy production rate and  some of them are new \blu{(cf. Theorem
\ref{teo_GC1},  Corollaries  \ref{cor_GC} and \ref{negramaro})}.
In addition, we discuss in detail the case of a 2-state model.
We point out that for a time periodic symmetric protocol, Gallavotti-Cohen duality relations for the entropy production rate have been experimentally verified in \cite{Schuller}  and theoretically analyzed in stochastic models in
\cite{Sinitsyn,Verley}.
Recently, in \cite{Ray}  a Gallavotti-Cohen duality relation for a not symmetric protocol has been obtained. Finally, we point out  that
  \cite{SW} provides a first  theoretical analysis of level 2.5 large deviations in periodically driven diffusion processes. }

\rosso{Our results  are a natural development  of the analysis  of level 2.5 large deviations for  time homogeneous  \blu{ergodic} Markov processes.} \blu{To explain this issue,  below  we recall  some fundamental  results  (in particular, below  we refer to time homogeneous processes).}
The celebrated papers by Donsker and Varadhan \cite{DV4} \rosso{have provided a
crucial contribution to the
 large deviation theory for ergodic Markov processes.}
One is typically interested on the long time
behavior of the process and three possible levels on which the large
deviations can be investigated have been identified: level 1, that
concerns the fluctuations of additive observables; level 2, that
concerns the fluctuations of the empirical measure; level 3, that
concerns the fluctuations of the empirical process.  These levels have
a hierarchical structure and
the large deviations on a lower level can be deduced by projection.
As the name implies, level 2.5 lies in between level 2 and level 3 and
concerns the joint fluctuations of the empirical measure and empirical
flow \blu{(or the  joint fluctuations of the empirical measure and empirical current)}. In the simple context of homogeneous continuous time Markov chains, the
empirical flow counts the numbers of jumps between pairs of states.
We emphasize that  in this case the rate functional for level 1 cannot be expressed in
closed form, for level 2 this is possible only in the reversible case,
while for level 3 the rate functional  is given by the specific relative
entropy with respect to the stationary process, that gives an explicit
but somehow abstract formula. On the other hand for level 2.5 there is
a simple explicit formula that covers both the reversible and non-reversible case, so  \blu{level 2.5 represents the lowest level admitting an explicit rate functional}.

A relevant motivation for the analysis of large deviations at level
2.5 comes from non-equilibrium statistical physics. Indeed, in this
context the current flowing through the system is a key observable and
exhibits rich and peculiar large deviation behavior, see e.g.\
\cite{noi,Lazarescu}. Moreover, the statistics of the entropy
production and the Gallavotti-Cohen symmetry cannot be described only
in terms of the empirical measure but require also the current
\cite{Chernyak3,Maes3,Maes4}.
From a purely probabilistic viewpoint,
the level 2.5 has been firstly investigated in \cite{Kesidis} in the
case of a two-state chain.  For a countable state space, the level 2.5
weak large deviation principle has been established in
\cite{Fortelle}.  In the same setting, the  large deviation
principle is proven in \cite{BFG1} \rosso{(and further analyzed in \cite{BFG_MPRF,BFG2})}, while the analogous result for
diffusion processes is obtained in \cite{Kusuoka}.
A more general setting with time dependent empirical measure and flow
is considered in \cite{KJZ,Re}. We also point out  that  recently some  thermodynamic uncertainty relation \cite{BaratoS} and some related
universal bound on current fluctuations \cite{Barato2,Barato3} have
been derived in \cite{Gingrich,Gingrich2} by using   the  level 2.5 large
deviation principle.
Finally, we refer to \blu{\cite{BFG_MPRF,che3}} for \blu{a further  discussion about level 2.5 for time-homogeneous Markov processes}.\\

\noindent
{\bf Outline of the paper}. In Section \ref{MR} we describe our main
assumptions on the continuous time Markov chains with time periodic
rates. In Section \ref{sec_LDP} we introduce the empirical measure,
flow, current and state our main large deviation principles
(cf. Theorems \ref{teo1}, \ref{teo2} and \ref{teo3}). In Section \ref{sec_EP} we discuss three forms of entropy production and in Section
\ref{sec_GC} we state the associated   Gallavotti-Cohen duality relations    (cf. Theorem
\ref{teo_GC1},  Corollaries  \ref{cor_GC} and \ref{negramaro}).  In Section \ref{sec_2stati} we apply our general
results to the case of   two states continuous time Markov chains with
time periodic rates. The rest of the paper is devoted
to the proof of our results. In particular, in Section
\ref{sec_preliminare} we collect some preliminary facts. In Section
\ref{sec_UB} we prove the upper bound for the LDP stated in Theorem
\ref{teo2} (cf. Eq. \eqref{eq_UB}), the convexity  \blu{and the lower-semicontinuity of  the LD rate functional  of Theorem \ref{teo2}}, while in Section \ref{sec_LB} we
prove the lower bound (cf. Eq. \eqref{eq_UB}) and the goodness of the \blu{same} rate functional. Theorem \ref{teo1}
follows easily from Theorem \ref{teo2} by contraction and therefore
the proof is omitted. The proofs of Theorems \ref{teo3} and Theorem
\ref{teo_GC1} are given in Sections \ref{asia} and \ref{pizza},
respectively.

\section{Continuous time Markov chains with time periodic rates}\label{MR}
We consider a continuous time Markov chain $\xi=(\xi_t)_{t \in \bb R_+} $ on a finite state space $V$,  with time periodic jump rates. We call $r(y,z;t)$ the jump rate from $y$ to $z$  at time $t$ and we assume that $r(\cdot,\cdot;t)= r(\cdot,\cdot; t+T_0)$ for some  $T_0>0$. To have a well defined process we  assume that
$r(y,z; \cdot)$ is a measurable,  locally integrable nonnegative function (see below). We also convey that  $r(x,x; t) \equiv 0$.

Roughly, the dynamics is defined as follows. Starting from a state $x$, the Markov chain spends at $x$ a random time $\tau_1$ such that
$$ \bb P( \tau_1 >t)= \exp \left\{ - \int _0 ^t r(x; s) ds \right\} \,,
$$
where
\[ r(x;s):= \sum _z r(x,z;s) \,.\]
Knowing that $\tau_1=t_1$, at time $t_1$ the Markov chain jumps to a new state $x_1$ chosen randomly with probability $r(x,x_1;t_1)/ r(x;t_1)$, afterwards it waits in $x_1$ a random time $\tau_2$  such that
$$ \bb P( \tau_2>t)= \exp \left\{ - \int _{t_1} ^{t_1+t} r(x_1; s) ds \right\} \,.
$$
Knowing that $\tau_2=t_2$, at time $t_2$ the Markov chain jumps to a new state $x_2$ chosen randomly with probability $r(x_1,x_2;t_2)/ r(x_1;t_2)$, and so on.

Above we have  not used the periodicity of the jump rates and indeed the  construction is common to all  time inhomogeneous Markov chains. Formally, a time inhomogeneous Markov chain can be seen as a piecewise-deterministic Markov process and its precise definition follows from the general construction in \cite{D}.  Indeed, we can introduce the continuous variable $s \in [0, +\infty)$ and describe the state of the system at time $t$ by $(\xi_t , s_t)$ where $s_t :=t$. Then the evolution in $V \times \bbR_+$ is described by a time homogeneous piecewise-deterministic Markov process  with formal generator $L$
\begin{equation}\label{ago}
 L f(x,s) = \partial _s f(x,s) + \sum _y r (x,y;s) \bigl[ f(y,s) - f (x,s) \bigr]\,.
\end{equation}
Following \cite{D}, to have a well defined operator one needs that the jump rates
$r (x,y;\cdot )$ are  measurable,  locally integrable nonnegative functions.
Due to \cite{D} time inhomogeneous Markov chains enjoy the strong Markov property.

%

\medskip

 We denote by $E$ the set of pairs $(y,z)$ such that $r(y,z; t)>0$ for all $t>0$, $y \not =z$.
 \rosso{We think of $(V,E)$ as a directed graph. Moreover,  we write $\cS_{T_0}$ for the set $ \bbR/ T_0 \bbZ$, i.e. for the set $[0,T_0]$ with periodic boundary conditions ($0$ and $T_0$ have to be identified).  }

\bigskip

\noindent
{\bf Assumptions}. \emph{Our  assumptions are the following:
\begin{itemize}
\item[(A1)]
 If $r(y,z;t)>0$ for some $t>0$, then $r(y,z;t)>0$ for all $t>0$;
 \item[(A2)]  \rosso{The  directed graph $(V,E)$ is  strongly connected;}
 \item[(A3)] The jump rates are nonnegative measurable functions such that
\begin{align}
&\max _{(y,z)\in E}  \sup _{t\in [0,T_0]} r(y,z; t) <\infty\,,\label{25aprile}\\
& \min _{(y,z)\in E}  \inf _{t\in [0,T_0]} r(y,z; t) >0\,.\label{1maggio}
\end{align}
\item[(A4)]
We assume that the set $\cD$ has zero Lebesgue measure, where    $\cD\subset \cS_{T_0}$ is the set of discontinuity points of the jump rates  $\cS_{T_0} \ni t\mapsto r(y,z;t)\in [0, \infty)$, as  $y,z$  vary in $V$.
\end{itemize}}

\bigskip
Assumption (A2) means that, given two distinct sites $y,z$ in $V$, there is a family of vertexes $x_0,x_1, \dots,x_n$
such that $x_0=y$, $x_n=z$ and $(x_i, x_{i+1})\in E$ for all $i=0,1,\dots, n-1$.

We point out  that assumption (A4) is used only to derive Lemma \ref{fuji}.

\bigskip

Trivially, the discrete time  process $\tilde \xi=(\tilde \xi_n)_{n \geq 0}$, with  $\tilde \xi_n:= \xi_{n T_0}$,   is a time homogeneous  Markov chain. We write $\tilde p(y,z)$, $y,z\in V$, for  its jump probabilities.
Since $V$ is finite and  \rosso{ $(V, E)$ is strongly connected},   $\tilde \xi$ admits a unique invariant distribution $\pi_0$, i.e. a unique probability measure $\pi_0$ on $V$ such that
\begin{equation}
  \label{invariante}
  \sum _{y \in V} \pi_0(x) \, \tilde p(x,y)
  = \sum _{y \in V} \pi_0(y) \, \tilde p(y,x)\qquad \forall \:x \in V\,.
\end{equation}
Note that the Markov chain $\tilde \xi$ starting with the invariant distribution $\pi_0$ is stationary, i.e. its law is invariant under time shifts (cf. Th.1.7.1 in \cite{N}). As a byproduct of this fact, the Markov property fulfilled by $\xi$ and the $T_0$-periodicity of the jump rates, one easily gets that   the Markov chain $\xi$ starting with initial distribution $\pi_0$ is $T_0$-stationary, i.e. its law is invariant under time translations  along times $T_0, 2T_0, 3T_0,\dots$.
In particular, when $\xi$ starts with distribution $\pi_0$, the law $\pi_t$ of  $\xi_t$ is a $T_0$-periodic function from $\bb R_+$ to the space  $\mc P (V)$ of probability measures on $V$.  We point out that $\pi_0$ is indeed the only initial distribution for which the  Markov chain $\xi$ starting at $\pi_0$  is $T_0$-periodic, hence we call the associated law
of $\xi=(\xi_t)_{t\geq 0}$ on  the space of c\`adl\`ag paths $D(\bbR_+;V)$ the \emph{oscillatory steady state} (sometimes, as in \cite{S}, this state is called \emph{nonequilibrium oscillatory state}, shortly NOS).

In what follows, we set \blu{$\pi:= \pi_t dt$. $\pi$ is  a nonnegative measure on $V \times \cS_{T_0}$ with total mass $T_0$ (shortly, $\pi \in \cM_{+,T_0} (V\times \cS_{T_0})$).}
\rosso{Given a probability measure $\nu$ on $V$,
we write $\bbP_\nu$ for the law of the Markov chain $(\xi_t)_{t\geq0}$  with initial distribution $\nu$, and we simply write  $\bbP_x$ if $\nu=\d_x$. The associated expectations are denoted by $\bbE_\nu$ and $\bbE_x$, respectively.}

\subsection{Graphical construction}\label{grafica}
We conclude by providing  a graphical construction of  the continuous time  Markov chain $(\xi_t)_{t\in \bbR_+}$, which will be useful in what follows.

To each $(y,z) \in E$ we  associate  a Poisson  process   of rate $\l(y,z)= \sup_{t\in [0,T_0] } r (y,z; t)$. We write
\[ \cT_{y,z}= \{ t^{(1)}_{y,z} < t^{(2)} _{y,z}< \cdots \}\]
for the jump times of the above  Poisson  process.
 Let us write
 \[\cT_y= \{ t^{(1)}_y < t^{(2)} _y< \cdots \}  \] for the superposition $\cup_z  \cT_{y,z}$. It is known that $\cT_y$ is a Poisson point process on $(0,\infty)$ with rate $\l(y):=\sum_z \l(y,z)$ (i.e. $\cT_y$ is the  set of  jump times of a Poisson process with rate $\l(y)$).
 Note that $\l(y) <\infty$ due to \eqref{25aprile}.
For each $(y,z)\in E$ consider also a sequence of i.i.d.  random variables $\cU_{y,z}=(U^{(k)}_{y,z}) _{k \geq 1}$ uniformly distributed on $[0,1]$.  The random objects given by $\cU_{y,z}$, $\cT_{y,z}$ with $(y,z)$ varying  in $E$, must be all independent.

Then the graphical construction is the following.  Suppose that $t=0$ or that the chain has been updated  at time $t$ and its state at time $t$ is  $y$. Let $s$ be the minimum of the set
$ \cT_y \cap (t,+\infty)$ and let $k,z$ be such that $s= t^{(k)}_{y,z}$ (they are well defined a.s.). Then  $s= t^{(k)}_{y,z}$ is an update time and the update is the following:
if  $U_k\leq \frac{ r(y,z; s)}{\l(y,z)}$ then we let   $\xi_s:=z$, otherwise  we let
$\xi_s:=y$. After the update the algorithm starts  afresh.

\section{Large deviation principles}\label{sec_LDP}
\subsection{Joint LDP for the empirical measure and flow}

\begin{definition}\label{pierpibello}
Given $T>0$, to each path $X \in D( \bb R_+; V)$  we associate the empirical measure $\bar \mu_T(X)\in \mc P(V) $ and the empirical flow $\bar Q_T(X) \in \bb R_+^E$ defined as
$$ \bar \mu_T(X)= \frac{1}{T} \int _0^T \delta_{X_t} dt \,, \qquad \bar  Q_T(X)(y,z)= \frac{1}{T}
\sum_{\substack{t \in [0,T]: \\ X_{t-} \not = X_t }} \mathds{1}\bigl(( X_{t-},  X_t)= (y,z) \bigr) \,.
$$
\end{definition}

  Let
$\Phi\colon \bb R_+ \times \bb R_+ \to [0,+\infty]$ be the function
defined by
\begin{equation}
  \label{Phi}
   \Phi (q,p)
   :=
   \begin{cases}
     \displaystyle{ q \log \frac qp - (q-p)}
     & \textrm{if $q,p\in (0,+\infty)$}\,,
     \\
     \;p  & \textrm{if $q=0$, $p\in [0,+\infty)$}\,,\\
     \; +\infty & \textrm{if $p=0$ and $q\in (0,+\infty)$.}
   \end{cases}
\end{equation}
For $p>0$, $\Phi( \cdot, p)$ is a nonnegative strictly convex function and is zero only at $q=p$. Indeed, since $\Phi(q,p)=\sup_{s\in \mathbb R}\left\{qs-p(e^s-1)\right\}$, $\Phi(\cdot, p) $   is the rate functional for the LDP of the sequence  $N_T/T$ as $T \to +\infty$, $(N_t)_{t \in \bb R_+}$
being a Poisson process with parameter $p$.

\medskip

  Given $t \in [0,T_0)$, let  $I_t\colon \mc P(V)\times \bb R_+^E \to
[0,+\infty]$ be the functional defined by
\begin{equation}
  \label{rfq}
  I _t (\bar \mu,\bar Q) :=
    \sum_{(y,z)\in E} \Phi \big( \bar Q(y,z), \bar \mu(y) r(y,z;t)\big)\,.
\end{equation}

\begin{theoremA}\label{teo1}   For each $x \in V$, \blu{by taking $T$ of the form $T=nT_0$ with $n$ integer,}
as $T\to+\infty$ the family of probability measures \rosso{$\{\bb
  P_x\circ (\bar\mu_T,\bar Q_T)^{-1}\}$} on $\mc P(V)\times \re$
  satisfies a large deviation principle with speed $T$ and good and convex rate functional $\bar I$ defined as
  \begin{equation}\label{aiuto}
  \bar I (\bar \mu,\bar Q)= \inf _{(\mu_t,Q_t)_{\rosso{t \in\cS_{T_0}} }}\frac{1}{T_0} \int _0^{T_0} I_t(\mu_t, Q_t) dt \,,
  \end{equation}
  where the infimum is taken among all measurable functions $\cS_{T_0} \ni t \to (\mu_t, Q_t)\in
  \mc P(V)\times \re$ such that
 \begin{equation}\label{siriano}
  \begin{cases}
   \partial _t \mu_t + \div Q_t  =0\,,\\
   \frac{1}{T_0}  \int _0^{T_0} \mu_t dt =\bar \mu\,,\\
   \frac{1}{T_0} \int _0^{T_0} Q_t dt = \bar Q\,.
   \end{cases}
   \end{equation}
\end{theoremA}
Theorem \ref{teo1}   follows easily  by contraction from  Theorem \ref{teo2} below, hence we omit the proof.

\smallskip
We give some comments on the notation used in Theorem \ref{teo1}. First, we recall that  the infimum of the empty set equals $+\infty$ by definition.
We  also recall that given   $A\in  \bbR_+^E$, the divergence $\div A: V \to \bbR$ is defined as
\begin{equation}  \label{ciocco}
\div A(y) = \sum_{z: (y,z) \in E} A(y,z) - \sum_{z: (z,y) \in E} A(z,y)\,.
\end{equation}
Below we will often use the convention that, given a  function \rosso{$B: E \to \bbR$, we set $B(y,z):=0$ if $(y,z) \not \in E$.} For example, due to this convention, we can rewrite \eqref{ciocco} as $\div A(y) = \sum_{z} A(y,z) - \sum_{z} A(z,y)$.
Finally,
the above \emph{continuity  equation }  $\partial _t \mu_t + \div Q_t  =0$ in \eqref{siriano} is thought of  in weak sense, i.e. (using the time $T_0$-periodicity)
\begin{equation}\label{zuccotto}
\int_0^{T_0} \sum _y \mu_s(y) \partial_s f(y,s) ds =  \int_0 ^{T_0} \sum _y \div Q_s(y) f(y,s) ds\,,
\end{equation}
for any $C^1$ function $f : V   \times \cS_{T_0}\to  \bbR$.
Here and in what follows, the $C^1$-regularity refers to time.

We finally observe that if $ \div \bar Q \neq 0$ then  $ \bar I (\bar
\mu,\bar Q)= +\infty$. Indeed by taking time average of the continuity
equation in \eqref{siriano} and using that $t\mapsto \mu_t$ is
\rosso{
defined on $\cS_{T_0}$ (and therefore can be thought of as a
$T_0$-periodic function),}
we deduce $\div \bar Q =0$.

\begin{remark}\label{parto}
By contraction one derives from Theorem \ref{teo1} both a LDP for the empirical measure $\bar \mu_T$ and a LDP  for the empirical flow $\bar Q_T$  \blu{(cf.  \cite{BFG_MPRF} for the corresponding contraction in the time--homogeneous case)}.
\end{remark}

\begin{remark}\label{parto2}
  By the goodness of the rate function $I$ in  Theorem~\ref{teo2},
  the infimum in \eqref{aiuto} is achieved whenever \eqref{siriano}
  admits a solution.  \blu{In particular, by the goodness of $I$, we have that   $\bar I (\bar \mu,\bar Q)= 0$
  if and only if there exists a pair $\mu=\mu_t dt$, $Q=Q_t dt$ solving \eqref{siriano} and such that $I(\mu, Q)=0$. As a byproduct with
  Remark    \ref{t-rex} below,  we conclude that
   $\bar I (\bar \mu,\bar Q)= 0$ if and only if $\bar\mu =\frac
  1{T_0}\int_0^{T_0} \pi_t \, dt$ and $\bar Q(y,z) = \frac
  1{T_0}\int_0^{T_0} \pi_t(y) r(y,z;t) \, dt$ for each $(y,z) \in E$. }\end{remark}

\subsection{Joint LDP for the extended empirical measure and flow}\label{2018}
We introduce the space  $\mc M_{+,T_0} ( V \times \cS_{T_0} )$ as   the family of nonnegative measures on $V \times \cS_{T_0} $ with total mass equal to $T_0$, and the space  $\mc M_+ (E \times \cS_{T_0}) $ as  the family of  nonnegative measures on $E \times \cS_{T_0}$ with finite total mass. Both spaces are endowed with the weak topology, i.e. $\nu_n \to \nu$  if and only if $\nu_n(f) \to \nu(f)$ for any bounded continuous function $f$ (by compactness,   continuous functions on $V \times \cS_{T_0}$ and on  $E \times \cS_{T_0}$   are automatically bounded).
  We will often use the trivial identifications   $\mc M_{+} ( V \times \cS_{T_0} ) \sim \mc M_{+} ( \cS_{T_0} ) ^V $ and $\mc M_+ (E \times \cS_{T_0}) \sim  \mc M_+ ( \cS_{T_0})^E  $, as in the following definition:

\begin{definition}
Given a positive integer $n$,  to each path $X \in D( \bb R_+; V)$ we associate the \emph{extended empirical measure} $\mu^{(n)} \in  \mc M_{+,T_0} ( V \times\cS_{T_0}) $ and the \emph{extended empirical flow} $Q^{(n)} \in \mc M_+ (E \times \cS_{T_0})$ defined by
\begin{align}
& \mu^{(n)}(x,dt )= \mu^{(n)} _t (x) dt  \;\;  \text{ where } \;\;\mu^{(n)} _t (x) := \frac{1}{n} \sum _{k=0}^{n-1}  \delta _{X_{t+k T_0}} (x) \,, \label{poretti}\\
& Q^{(n)}(y,z,B)= \frac{1}{n } \sum _{k=0}^{n-1}  \, \sum _{ \substack{ t\in B+k T_0\,:\\  X_{t^-} \not= X_t}} \mathds{1}\bigl(  (X_{t-},X_t)=(y,z) \bigr)\,,\label{moretti}
\end{align}
where $B$ is a generic Borel subset $B \subset (0,T_0]$ (in the above formulas we have used the natural parametrization of  $\cS_{T_0}$   by $(0,T_0]$).
\end{definition}

\medskip
 We can identify functions   $f: V \times \cS_{T_0} \to \bbR$ with functions $f: V \times \bbR_+\to \bbR $ which are $T_0$-periodic in the time variable. In what follows, when we say that $f: V \times \bbR_+\to \bbR $ is $T_0$-periodic or $C^k$ we always  mean in the time variable. Similar considerations hold for functions $f: E\times \cS_{T_0} \to \bbR$. By means of this identification, we can rewrite \eqref{poretti} and \eqref{moretti} as
 \begin{align}
& \mu^{(n)} (f) = \frac{1}{n}\int_0^{nT_0} f( X_t, t) dt\,,\qquad  \qquad\;\;  \qquad  f: V \times \cS_{T_0} \to \bbR\,,\label{roma_as}  \\
&Q^{(n)}(g)=\frac{1}{n} \sum _{\substack{ t \in [0, nT_0] : \\
X_{t^-}\not = X_t}} g  (X_{t-},X_t, t)\,, \qquad \qquad    g: E \times \cS_{T_0} \to \bbR \,,\label{juve}
 \end{align}
 where $f,g$ are bounded and measurable.

 \medskip

 To simplify the notation from now on we set
 \begin{equation}\label{ho_fame}
 \cM _*:= \mc M_{+,T_0}( V \times \cS_{T_0} )\times \mc M_+ (E \times \cS_{T_0})\,.
 \end{equation}

 \begin{definition}\label{waffel}
 We introduce the subset $\L\subset \cM_*$ given by the pairs $(\mu,Q)\in \cM_*$ such that:
 \begin{itemize}
 \item[(i)]   $\mu= \mu_t dt$ with
$\mu_t(V)=1$ for almost every $t \in \cS_{T_0}$;
\item[(ii)] $Q= Q_t dt $;
\item[(iii)]  $\partial_t \mu_t+ \div Q_t=0$ \rosso{weakly};
\item[(iv)] for almost every $t\in \cS_{T_0}$ it holds: $\mu_t(y) =0 \Rightarrow Q_t(y,z)=0$ for all $(y,z)\in E$.
\end{itemize}
 \end{definition}
\begin{theoremA}\label{teo2}
Given $x \in V$ the family $\left	\{ \bb P_x \circ (\mu^{(n)},Q^{(n)})^{-1}\right\}_{n \geq 1}$ of probability measures on
$\cM_*$ satisfies   a   large deviation principle with speed $n$ and good and convex rate functional $I$ defined as
\begin{equation}\label{pakistano}
I(\mu,Q):= \begin{cases}
\int _0^{T_0} I_t(\mu_t,Q_t) dt  & \text{ if }  (\mu,Q)\in \L\,,\\
+\infty & \text{ otherwise}\,.
\end{cases}
\end{equation}
\end{theoremA}
\blu{The proof of the above theorem is given in Sections \ref{sec_UB} and \ref{sec_LB}.}

\smallskip

We recall that the above LDP means that, for any $\cC \subset \cM_*$ closed and any $\cG \subset \cM_*$ open, it holds
\begin{align}
& \varlimsup_{n\to \infty} \frac{1}{n} \log \bbP_x \bigl( (\mu^{(n)},Q^{(n)}) \in \cC ) \leq - \inf _{(\mu,Q)\in \cC} I(\mu,Q)\,, \label{eq_UB}\\
& \varliminf_{n\to \infty} \frac{1}{n} \log \bbP_x \bigl( (\mu^{(n)},Q^{(n)}) \in \cG ) \geq - \inf _{(\mu,Q)\in \cG} I(\mu,Q)\,.
\label{eq_LB}
\end{align}

\begin{remark}\label{t-rex}
We point out that $I(\mu,Q) =0$ if and only if $\mu = \pi_t\,dt$ and $Q= Q_t dt$ with
$Q_t(y,z)= \pi_t(y) r(y,z;t)$, $(y,z)\in E$. \rosso{Indeed, by the properties of the function $\Phi$ stated after \eqref{Phi}, $I_t(\mu_t, Q_t)=0$ if and only if $Q_t (y,z)= \mu_t(y) r(y,z;t)$ for any $(y,z) \in E$. By Proposition \ref{ananas} and the continuity equation, this last property is satisfied  for almost all $t$ only when $\mu = \pi_t\,dt$ and $Q= Q_t dt$ with
$Q_t(y,z)= \pi_t(y) r(y,z;t)$ for any $(y,z)\in E$.}
\end{remark}
Since $\bar \mu_{n T_0} (\cdot)= \frac{1}{T_0} \mu^{(n)}(\cdot,
\cS_{T_0} ) $ and $\bar Q_{n T_0} (\cdot)= \frac{1}{T_0}
Q^{(n)}(\cdot, \cS_{T_0}) $, Theorem \ref{teo1} with $T$ of the form $nT_0$ follows from Theorem
\ref{teo2} by applying the contraction principle.
\subsection{LDP for currents}
Recalling that $E$ denotes the set of ordered edges of $V$ with strictly positive jump rates, we let $E_s:=\{ (y,z) \in V \times V\,:\, (y,z) \in E \text{ or } (z,y) \in E\}$ be the symmetrization of $E$ in $V \times V$. We denote by $\bbR_a^{E_s}$  the family of functions $\bar J: E_s \to \bbR$ which are antisymmetric, i.e. $\bar J(y,z)=- \bar J(z,y)$ $\forall (y,z) \in E_s$.
\begin{definition}\label{loribello}
Given $T>0$, to each path $X \in D( \bb R_+; V)$  we associate the \emph{empirical current}   $\bar J_T(X) \in \bb R_a^{E_s}$ defined as
\begin{equation}
 \bar J_T(X)(y,z)
 = \frac{1}{T}
\sum_{\substack{t \in [0,T]: \\ X_{t-} \not = X_t }}\Big[ \mathds{1}\bigl(( X_{t-},  X_t)= (y,z) \bigr) -
\mathds{1}\bigl(( X_{t-},  X_t)= (z,y) \bigr)
\Big]
\end{equation}
for any $(y,z) \in E_s$.
\end{definition}

To introduce  the \emph{extended empirical current} we denote by $\cM_a (E_s \times \cS_{T_0})$  the space of signed measures  $J$ on $E _s\times \cS_{T_0}$ which are antisymmetric in $E_s$ (i.e. $J(y,z,A)=-J(z,y,A)$ for any $A \subset \cS_{T_0} $ measurable) and have  finite total variation  (i.e. $J$ can be written as difference
of two measures in $\cM_+ ( E_s \times \cS_{T_0})$). $\cM_a (E _s\times \cS_{T_0})$ is endowed with the usual weak topology, i.e. $\nu_n \to \nu$ if and only if $\nu_n(f)\to \nu(f)$ for any continuous function on $E _s\times \cS_{T_0}$.
\begin{definition}\label{cicciobello}
Given $T>0$, to each path $X \in D( \bb R_+; V)$  we associate the \emph{extended empirical current}   $J^{(n)} (X) \in \cM_a (E_s \times \cS_{T_0})$ defined as
\begin{equation}
J^{(n)}(y,z,B)= \frac{1}{n } \sum _{k=0}^{n-1}  \, \sum _{ \substack{ t\in B+k T_0\,:\\  X_{t^-} \not= X_t}}\Big[ \mathds{1}\bigl(  (X_{t-},X_t)=(y,z) \bigr)-\mathds{1}\bigl(  (X_{t-},X_t)=(z,y) \bigr)\Big]
\label{moretti_J}
\end{equation}
for any $B \subset [0,T_0)$ measurable.
\end{definition}

We introduce  the continuous map
\begin{equation}\label{gb}
\cJ : \cM_+ ( E \times \cS_{T_0}) \to \cM _a ( E_s \times \cS_{T_0}) \end{equation}
defined as
\[ \mathcal J(Q)(y,z,A):=Q(y,z,A)-Q(z,y,A)\,, \]
for any $ (y,z) \in E_s$ and $ A\subset \cS_{T_0}$   measurable,
with the convention that  $Q(y',z',A):=0$ if $(y',z') \not \in E$. Trivially, the following relation holds between the extended empirical flow and current:
\begin{equation}
\cJ\bigl( Q^{(n)} \bigr)= J^{(n)}\,.
\end{equation}
As a consequence, from the contraction principle and the joint LDP for  $(\mu^{(n)}, Q^{(n)})$  given in Theorem \ref{teo2}, we get that
 a joint LDP  holds for $(\mu^{(n)}, J^{(n)})$  with speed $n$ and   good and convex rate functional $\hat I(\mu,J)$ given by
 \begin{equation}\label{arancione}
\hat I(\mu,J):= \inf_{Q: \mathcal J(Q)=J}I(\mu,Q)\,.
\end{equation}
It turns out that the  above variational problem expressing the new rate functional $\hat I$ can be exactly solved, thus leading to Theorem \ref{teo3} below. In order to state this theorem, we need a preliminary definition:
\begin{definition}\label{deficit}
The set $\L_a$ is given by the pairs  $(\mu,J)\in  \mc M_{+,T_0}( V \times \cS_{T_0} )\times \mc M_a (E_s \times \cS_{T_0})$ such that
\begin{itemize}
 \item[(i)]   $\mu= \mu_t dt$ with
$\mu_t(V)=1$ for almost every $t \in \cS_{T_0}$;
\item[(ii)] $J= J_t dt $;
\item[(iii)]  $\partial_t \mu_t+ \div J_t=0$ where $\div J_t(y)= \sum_{(y,z) \in E_s} J_t(y,z)$;
\item[(iv)] for almost every $t\in \cS_{T_0}$ it holds: $\mu_t(y) =0 \Rightarrow J_t(y,z)\leq 0$ for all $(y,z)\in E_s$;
\item[(v)] $J_t(y,z)\geq 0$ if $(y,z)\in E$ and $(z,y)\not \in E$, while $J_t(y,z)\leq  0$ if $(y,z)\not \in E$ and $(z,y) \in E$
\end{itemize}
\end{definition}
We recall that the continuity equation in Item (iii) has to be thought in its weak form.
To  state the   joint  LDP for $(\mu^{(n)}, J^{(n)})$ we introduce also  the function
${\Psi: \bb R\times \bb R\times \bb R_+ \mapsto [0, +\infty]}$ given
by\footnote{\rosso{This formula corresponds to  \cite[Eq. (6.3)]{BFG2}, apart the correction of a typo there in the case $a=0$.}}
 \begin{equation}
\Psi ( u, \bar u; a):=
\begin{cases}
u \left[ {\rm arcsinh\,} \frac{u}{a} - {\rm arcsinh\,} \frac{\bar u
  }{a} \right] \\
\qquad - \left[ \sqrt{a^2+u^2} -
\sqrt{a^2 + \bar{u}^2 }\right] & \text{ if } a>0\,,\\
\Phi(| u|, |\bar u|)  & \text{ if } a=0\,.\\
\end{cases}
\end{equation}
 We recall that   ${\rm arcsinh\,}(x)  =\log[ x +\sqrt{x^2+1}]$.

Finally, for the theorem below,  we recall that $r(y,z;t):=0$ if $(y,z)\not\in E$.
\begin{theoremA}\label{teo3}
Given $x \in V$ the family $\left	\{ \bb P_x \circ (\mu^{(n)},J^{(n)})^{-1}\right\}_{n \geq 1}$ of probability measures on
\[
 \mc M_{+,T_0}( V \times \cS_{T_0} )\times \mc M_a (E_s \times \cS_{T_0})
 \]
  satisfies   a   large deviation principle with speed $n$ and   good and convex rate functional $\hat I$ given by
\begin{equation}\label{piscina}
\hat I(\mu,J)=
\begin{cases}
\int_0^{T_0}I_t(\mu_t, Q^{J,\mu}_t)\,dt& \text{ if } (\mu,J)\in \L_a \,,\\
+ \infty   & \text{ otherwise}\,,
\end{cases}
\end{equation}
where
\begin{equation}\label{domus}
Q^{J,\mu}_t(y,z)=\frac{J_t(y,z)+\sqrt{J_t^2(y,z)+4\mu_t(y)\mu_t(z)r(y,z;t)r(z,y;t)}}{2}\,.
\end{equation}
Moreover,  given  $(\mu,J)\in \L_a$,  the rate functional $\hat I(\mu,J)$ can be rewritten as
\begin{equation}\label{cuore}
\hat I(\mu,J)= \frac{1}{2}\sum_{(y,z)\in E_s}\int_0^{T_0} \Psi\left( J_t (y,z), J_t^\mu (y,z); a_t ^\mu(y,z)\right)dt \,,
\end{equation}
where
\begin{align*}
& J_t^\mu (y,z):=\mu_t(y) r(y,z;t)- \mu_t (z) r(z,y;t)\,,\\
& a_t ^\mu(y,z):= 2\sqrt{\mu_t(y) \mu_t(z) r(y,z;t) r(z,y;t)}\,.
\end{align*}
\end{theoremA}
The proof of Theorem \ref{teo3} is given in Section \ref{asia}.

\begin{remark}\label{nuovo}
  By contraction, Theorem~\ref{teo3} implies a joint LDP for the
  empirical measure and current. The corresponding convex  rate functional
  $\overline{\!\hat I}\, \colon \mc P(V)\times \bb R_a^{E_s} \to
  [0,+\infty]$ is given by
\begin{equation}\label{esilio}
\overline{\!\hat I}\,(\bar \mu, \bar J) = \inf _{(\mu,J)}\frac{1}{T_0} \hat I(\mu, J) \,,
\end{equation}
where the infimum is taken among all pairs $(\mu,J)$ in $\L_a$ such that
$\frac{1}{T_0}  \int _0^{T_0} \mu_t dt =\bar \mu$ and $ \frac{1}{T_0} \int _0^{T_0} J_t dt = \bar J$, where $\mu= \mu_t dt $ and $J= J_t dt $.
\end{remark}


\section{Stochastic entropy flow}\label{sec_EP}
In this section we assume that $(y,z) \in E$ if and only if $(z,y)\in
E$, i.e. $E=E_s$.

One usually  defines the   \emph{fluctuating entropy flow} on the time interval
$\left[0,nT_0\right]$ as   the Radon-Nikodym derivative
 \begin{equation}\label{eq:ep}
 \sigma_{nT_0}\left[X\right]=\log   \left.\frac{d\mathbb{P}}{d\left(\mathbb{P}^{B}\circ R_{nT_{0}}\right)}\right|_{\left[0,nT_{0}\right]}\bigl( (X_{s})_{s\in [0,nT_{0}]}\bigr)\, ,\end{equation}
where  $\left.\mathbb{P}\right|_{\left[0,nT_{0}\right]}$ is the law
on $D\left(\left[0,nT_{0}\right];V\right)$ of the continuous time Markov chain
with rates $r(\cdot,\cdot;t)$ and  some initial distribution   $\mu_{0}$ and $\left.\mathbb{P}^{B}\right|_{\left[0,nT_{0}\right]}$ is the law on $D\left(\left[0,nT_{0}\right];V\right)$ of another   continuous time Markov chain
with rates  $r^{B}(\cdot,\cdot;t)$, and some initial distribution  $\mu'_{0}$. Then the measure $\mathbb{P}^{B}\circ R_{nT_{0}}$ is
the pushforward measure of the law $\mathbb{P}^{B}$  in the time window $[0, nT_0]$  by $R_{nT_0}$,  $R_{nT_0}$ being
the pathways time reversal \blu{in the time window $[0,nT_{0}]$}. \rosso{Of course, definition    \eqref{eq:ep} restricts to the case when the
 Radon-Nikodym derivative is   well defined.} We further restrict to the case of $T_0$-periodic rates $r(\cdot,\cdot;t)$ and $r^B(\cdot,\cdot;t)$.
 Below
we will consider only three peculiar choices of  rates  $r^B(\cdot, \cdot;t)$:
the naive reversal (cf. Subsection \ref{sec_a1}),  the reversed protocol first used in \cite{Crooks} (cf. Subsection \ref{sec_a2}) and the  dual reversed protocol first considered for  time inhomogeneous processes in \cite{Crooks} (cf. Subsection \ref{sec_a3}).
The fact that by playing with different choices  of the backward process (i.e. the rates $r^B$  here) we find different physical quantities (excess heat, housekeeping heat, \rosso{total heat}, phase-space contraction,...) has been pointed out in \cite{Chernyak, che1} for diffusion processes and in \cite{Harris}  for pure jump Markov processes.
\rosso{We recall that the excess heat and the housekeeping heat  were first introduced by Oono and Paniconi  \cite{Oono}:  the excess heat   measures the non stationarity of the process, while the housekeeping heat measures the distance of the process from the ``instantaneous'' reversibility.  These two quantities permit to obtain a  refinement  of the second law of thermodynamics (see \cite{Chernyak,che1,Harris}).  The total heat is then the sum of the excess heat and the housekeeping heat. On the other hand, the phase-space contraction is   a quantity characterizing the irreversibility of a deterministic dynamical system \cite{GC,Ru}.  For   diffusion processes  see \cite{che1} for the fluctuation relation associated to the
 phase-space contraction and to its generalization to   obtain the fluctuation relation  of the finite time Lyapunov exponents.}

\medskip

We point out that \eqref{eq:ep} implies
directly the finite time fluctuation relation
\begin{equation}
\mathbb{P}\left(\sigma_{nT_{0}}\left[X\right]\in [\s,\s+d\s)\right)\exp\left(-\sigma\right)=\mathbb{P}^{B}\left(\sigma_{nT_{0}}^{B}\left[X\right]\in [-\sigma, -\s + d\s)\right),\label{eq:RFtf}
\end{equation}
with the backward entropy flow $\sigma_{nT_{0}}^{B}\left[X\right]$
 defined by
\begin{equation}
\sigma_{nT_{0}}^{B}\left[X\right]=\log\left.\frac{d\mathbb{P}^{B}}{d\left(\mathbb{P}\circ R_{nT_{0}}\right)}\right|_{\left[0,nT_{0}\right]}\bigl( \left(X_{s}\right)_{s\in\left[0,nT_0\right]}\bigr)\,.\label{eq:ep-1}
\end{equation}
By \blu{using Eq. \eqref{diderot}} in Section \ref{sec_RN} and the periodicity of the rates, we get
 \begin{equation} \begin{split}
\sigma_{nT_{0}}\left[X\right]& =\log \frac{\mu_{0}\left(X_{0}\right)}{\mu_0'\left(X_{nT_0}\right)} +
 \sum _{\substack{s \in (0,nT_0]:\\ X_{s-} \not = X_s}}
\log \frac{r\left(X_{s^{-}},X_{s};s\right)}{r^{B}\left(X_{s},X_{s^{-}};T_{0}-s\right)}\\
&-\int_{0}^{nT_{0}}ds\left[r\left(X_{s};s\right)-r^{B}\left(X_{s};T_{0}-s\right)\right]\,.\label{eq:EF-1}
\end{split}
\end{equation}

%
%
%
We point out that, since $V$ is finite, the boundary term ${\rm b.t.}= \log \frac{\mu_{0}\left(X_{0}\right)}{\mu_0'\left(X_{nT_0}\right)} $ will not play any role in the large deviation limit. We now provide three examples of entropy flows relevant in statistical physics, and show that - apart negligible boundary terms - they can be expressed by contraction  from the empirical extended  measure and/or flow.
\subsubsection{Entropy flow from naive reversal}\label{sec_a1}
We take the \emph{identity reversal} $r^{B}(y,z;t):=r(y,z;t)$ and we write $\sigma_{nT_{0}}^{\rm naive}$ for the associated entropy flow.
We define the functional $S_{\rm naive}(\mu,Q)$ on the space $\cM_*$  introduced in \eqref{ho_fame} as follows:
\begin{equation}\label{eq:SNR}
\begin{split}
S_{\rm naive}(\mu,Q)
:&= \sum _{(y,z)\in E}\int Q(y,z,ds) \log  \frac{r(y,z;s)}{r(z,y;T_{0}-s)}\\
&-\sum_y \int \mu(y,ds) [ r(y;s)-r(y,T_0-s)]\,.
\end{split}
\end{equation}
Then the entropy flow fulfills the identity
\begin{equation}\label{eq:SNR2}
\frac{1}{n}\sigma_{nT_{0}}^{\rm naive}=\frac{{\rm b.t.}}{n}+S_{\rm naive}\bigl(\mu^{(n)},Q^{(n)}\bigr)\,.
\end{equation}
\subsubsection{Total entropy flow from reversed protocol}\label{sec_a2}
As rates $r^B$ \blu{(denoted here by $r^{\rm R}$)}
 we  choose  the \emph{reversed protocol}, i.e. we take the rates   \begin{equation}\label{crema}
r^{\rm R}(y,z;t):=r(y,z;T_{0}-t)\,.
\end{equation} The resulting entropy flow, usually called   \emph{total entropy flow}, will be denoted by $\sigma_{nT_{0}}^{\text{tot}}$. Defining the functional $S_{\text{tot}}$ as
\begin{equation}\label{eq:SPR}
S_{\text{tot}}(Q)
:= \sum _{(y,z)\in E} \int Q(y,z,ds) \log  \frac{r(y,z;s)}{r (z,y;s)}
\end{equation}
for $Q \in  \mc M_+ (E \times \cS_{T_0})$,
we get
 \begin{equation}
\frac{1}{n}\sigma_{nT_{0}}^{\text{tot}} =\frac{{\rm b.t.}}{n}+S_{\text{tot}}\bigl( Q^{(n)}\bigr)\,.\label{eq:SPR2}
\end{equation}
The above   entropy production  has been investigated in
\cite{Proesmans2,Ray} for  time periodic processes.

\subsubsection{Entropy flow in excess }\label{sec_a3}
 Given $t$ we write $w_t$ for the  \emph{accompanying distribution} (cf. \cite{Hanggi}), which is defined as the
unique invariant distribution  of the time homogeneous (and continuous time) Markov chain
 with frozen jump rates $r(\cdot, \cdot;t)$. Due to our assumptions (A1) and (A2) the distribution $w_t$ is well defined, and moreover it is strictly positive on each state of $V$.

As rates $r^B$ \blu{(denoted here by $r^{\rm DR}$)}
 we then choose  the \emph{dual reversed protocol}:
\begin{equation}
r^{\rm DR} (y,z;t)=w_{T_{0}-t}^{-1}(y) r(z,y;T_{0}-t )w_{T_{0}-t}(z)\,.\label{eq:DR}
\end{equation}
The resulting entropy flow, denoted by $\sigma_{nT_{0}}^{\rm ex}$ and called \emph{excess entropy flow}, is related to    the excess heat discussed in  \cite{HS,Oono}. By the  invariance of $w_t$, from \eqref{eq:EF-1} one easily gets the simplified expression
 \begin{equation}\label{eq:EFDR}
\sigma_{nT_{0}}^{\rm ex}[X] = {\rm b.t.}+\sum
_{\substack{s \in (0,n T_0]:\\ X_{s-} \not = X_s}}
\log \frac{w_{s}(X_{s})}{w_{s}\bigl(X_{s^{-}}\bigr)}\,.\end{equation}
At the cost of a boundary term (irrelevant in the LD limit) we find
\begin{equation}\label{eq:EFDR-1}
\sigma_{nT_{0}}^{\rm ex}\left[X\right]={\rm b.t.}'-\int_{0}^{nT_{0}}ds\left[\partial_{s}\left(\log w_{s}\right)\right]\left(X_{s}\right)\,,\end{equation}
which is the quantity considered in the time periodic set-up by Schuller et al. in \cite{Schuller}.
By defining
the functional
\begin{equation}\label{eq:SPR-1}
S_{\rm ex}(\mu):=- \sum _y \int \mu(y,ds)  \partial_s
\log\bigl( w_{s} (y) \bigr)
\end{equation}
for $\mu \in  \mc M_{+,T_0}( V \times \cS_{T_0} )$, we can write
\begin{equation}
\frac{1}{n}\sigma_{nT_{0}}^{\rm ex} =\frac{b.t.'}{n}+S_{\rm ex}\bigl(\mu^{(n)}\bigr)\,.
\end{equation}

\section{Gallavotti-Cohen duality relations}\label{sec_GC}
As in Section \ref{sec_EP}  we assume  that $(y,z) \in E$ if and only if $(z,y)\in
E$, i.e. $E=E_s$.

\medskip

We recall that, given $(\mu,Q)\in \cM_*$,
it holds $I(\mu,Q)=+\infty$ if $(\mu,Q)\not \in \L$ \rosso{(see  \eqref{ho_fame} and \eqref{pakistano})}. Hence, for the
analysis of   Gallavotti-Cohen duality relations, we restrict to
$(\mu,Q)\in\L$.

\begin{definition}
Given $(\mu,Q)\in \L$, with $\mu = \mu_t dt $ and $Q=Q_t dt $, we define the transformed  element $(\theta \mu, \theta Q)\in \L$ as $\theta \mu= (\theta \mu_t) dt $, $ \theta Q= (\theta Q_t) dt $ where
\begin{equation}
\theta \mu_t:= \mu_{T_0-t}\,, \qquad \theta Q_t (y,z):= Q_{T_0-t} (z,y)\,.
\end{equation}
\end{definition}
It is simple to check that $(\theta \mu, \theta Q)$ is indeed an element of $\L$.

\medskip
 In what follows we write  $I(\mu,Q; r)$ for the joint LD  rate functional of Theorem \ref{teo2} referred to the Markov chain with jump rates $r(y,z;t)$.
 Similarly, we add the reference to the jump rates in the entropy production functions by writing $S_{\rm naive} (\mu,Q;r)$, $S_{\rm tot} (Q;r)$ and $S_{\rm ex} (\mu;r)$
  (recall the notation introduced in Sections \ref{sec_a1}, \ref{sec_a2} and \ref{sec_a3}). By means of the contraction principle
 one derives from Theorem \ref{teo2} the LDP for the entropy production functions   $S_{\rm naive}(\mu,Q;r)$,
 $ S_{\text{tot}}(Q;r)$ and $S_{\rm ex}(\mu;r)$ with LD functionals given respectively by
\begin{align*}
\cI _{\rm naive}(s;r) &= \inf \{ I(\mu,Q;r)\,:\,  S_{\rm naive}(\mu,Q;r) =s\}\,,\\
\cI _{\rm tot}(s;r) &= \inf \{ I(\mu,Q;r)\,:\,   S_{\rm tot} (Q;r)=s\}\,,\\
\cI _{\rm ex}(s;r) &= \inf \{ I(\mu,Q;r)\,:\,   S_{\rm ex} (\mu;r) =s\} \,.
\end{align*}

 \begin{theoremA}\label{teo_GC1} For any $(\mu,Q)\in \L$ we have the following level 2.5  Gallavotti-Cohen  duality relations:
\begin{align}
 I(\theta \mu, \theta Q; r)& =I(\mu,Q; r)  +S_{\rm naive} (\mu,Q;r)\,,\label{uva1} \\
 I(\theta \mu, \theta Q; r^{\rm R})&=I(\mu,Q; r)+S_{\rm tot} (Q;r)   \,,\label{uva2}\\
 I(\theta \mu, \theta Q; r^{\rm DR})& =I(\mu,Q; r) +S_{\rm ex} (\mu;r)      \,.\label{uva3}
\end{align}
Moreover, for any real $s$
we have  by contraction  the following  Gallavotti-Cohen  duality relations:
\begin{align}
\cI _{\rm naive}(-s;r) &= \cI _{\rm naive}(s;r) +s \,,\label{rio1}\\
\cI _{\rm tot}(-s; r^{\rm R}) &=  \cI _{\rm tot}(s;r )+s  \,,\label{rio2}\\
\cI _{\rm ex}(-s;r^{\rm DR}  ) & = \cI _{\rm ex}(s;r)+s\,.\label{rio3}
\end{align}
\end{theoremA}
The above duality relations are new with exception of \eqref{rio2} which appears also in \cite{Ray,Sinitsyn}.
The proof of Theorem \ref{teo_GC1} is given in Section \ref{pizza}.

\smallskip

If we have a time symmetric protocol, i.e. $r\left(y,z;T_{0}-t\right)=r\left(y,z;t\right)$,  then the naive entropy flow  and the total entropy flow  are identical and the  duality relations  \eqref{rio1} and \eqref{rio2} become identical.
 If the accompanying distribution satisfies  the instantaneous detailed balance such that the relation \eqref{eq:DR} becomes
$r^{DR}\left(y,z;t\right)=r\left(y,z;T_{0}-t\right)$, then the excess entropy flow  and the total entropy flow  are identical. In particular,  the duality relations \eqref{rio2} and \eqref{rio3}  become  identical.
Finally, we point out   that   in \cite{Schuller} the  Gallavotti-Cohen relation has been  experimentally checked
in a  context where the two previous situations  both take place, hence in that context the three duality relations \eqref{rio1}, \eqref{rio2} and \eqref{rio3}   are identical.

\medskip

By the contraction principle,
the duality relations in  Theorem \ref{teo_GC1}   imply  some analogous relations for the extended current. To  this aim,
we define   $\theta J_t(y,z)= J_{T_0-t}(z,y)= - J_{T_0-t}(y,z)$ and write
$\hat I(\mu,J; r)$ for the LD rate functional  $\hat I(\mu,J)$ of $(\mu^{(n)}, J^{(n)})$ with jump rates $r(\cdot,\cdot;\cdot)$ \rosso{(see \eqref{arancione} and \eqref{cuore})}.
 In particular,
one  derives the following corollary:
\begin{corollary}\label{cor_GC} For any $(\mu,J)\in \L_a$ it holds
\begin{align}
 \hat I(\theta \mu, \theta J;  r^{\rm R} ) -\hat I(\mu,J; r)&= \frac{1}{2} \sum_{(y,z)\in E}\int_0^{T_0}J_t(y,z)\log\frac{r(y,z;t)}{r(z,y;t)}\,dt\,, \label{luci1}\\
 \hat I(\theta \mu, \theta J;  r^{\rm DR} ) -\hat I(\mu,J; r)&=S_{\rm ex} (\mu;r)  \,. \label{luci2}
\end{align}

\end{corollary}
\blu{
\begin{proof} Recall the map $\cJ$ defined in \eqref{gb} and note that
\begin{equation}\label{crepe1}
\cJ(Q)=J \Longrightarrow
\cJ(\theta Q)= \theta J\,.
\end{equation}
Observe also that,  if $(\mu,Q)\in \L$ is such that $\cJ(Q)=J$, then
\begin{equation}\label{crepe2}
S_{\rm tot} (Q;r)=\frac{1}{2} \sum_{(y,z)\in E}\int_0^{T_0}J_t(y,z)\log\frac{r(y,z;t)}{r(z,y;t)}\,dt\,.
\end{equation}
The duality relation \eqref{luci1} then follows  from \eqref{arancione} by taking the infimum in both sides of  \eqref{uva2} among all $Q$ with
$(\mu,Q)\in \L$ and $\cJ(Q)=J $, and by using \eqref{crepe1} and \eqref{crepe2}. The duality relation \eqref{luci2} follows by the same procedure applied to  \eqref{uva3}.
\end{proof}
}
We remark that as  the naive entropy flow \eqref{eq:SNR}  cannot be expressed (up to boundary terms) as contraction of the extended empirical measure and current,  there is no version of Corollary \ref{cor_GC}   (i.e. with extended current) for the duality relation \eqref{uva1}.
Finally, by applying once again the contraction principle to Corollary \ref{cor_GC}  we get    other duality relations (we omit the proof since simple):
\begin{corollary}\label{negramaro} It holds
\begin{align}
  \blu{I_{c}(\theta J;  r^{\rm R} ) -I_{c}(J; r)}&= \frac{1}{2} \sum_{(y,z)\in E}\int_0^{T_0}J_t(y,z)\log\frac{r(y,z;t)}{r(z,y;t)}\,dt\,,\label{luci3}\\
\blu{I_{m} (\theta \mu; r^{\rm DR} ) -I_{m}(\mu; r)}&=S_{\rm ex} (\mu;r)\,, \label{luci4}
\end{align}
\blu{where $I_{c}$ and $I_{m}$ denote the LD rate functionals of the extended empirical current and the extended empirical measure, respectively.}
\end{corollary}
  The first relation is the Gallavotti-Cohen relation for the LD rate functional of the extended empirical current only and the second  relation  is a level 2-duality relation for the LD rate functional of the extended empirical density only. We are not aware of previously derived relation of the type of  (\ref{luci4})  even in time homogeneous set-up.

Finally we  point out  that the LD rate functional  $\bar I (\bar \mu,\bar Q)$   (cf. \eqref{aiuto}) and  $\overline{\!\hat I}\,(\bar \mu, \bar J)$  (cf. \eqref{esilio})  do not satisfy duality relations resulting from   a naive contraction of  the relations in Theorem \ref{teo_GC1}.  Indeed,  the three  entropy flows  cannot be expressed as contraction of the empirical measure and empirical flow/current (recall  Definitions  \ref{pierpibello} and  \ref{loribello}).

\section{Two state systems}\label{sec_2stati}
We consider the simplest possible system, that is a two state ($V=\{0,1\}$) chain.  In this case the model is completely
determined by the two periodic functions $r_t(0,1)$ and $r_t(1,0)$ that fix the jump rates (for simplicity of notation, sometimes the  time variable $t$ will appear as subindex in the rates). Even if elementary, this framework has however interesting and non trivial physical applications.

\rosso{We list some relevant examples:}

\begin{itemize}
\item
In \cite{Proesmans2}  we have a  quantum dot  with one single active energy level periodically modulated that
corresponds to a two state Markov chain with rates
\begin{equation}
\begin{cases}
r_{t}(0,1)=\frac{\varGamma}{1+\exp(x_{t})}\,,\\
r_{t}(1,0)=\frac{\varGamma\exp(x_{t})}{1+\exp(x_{t})}\,,
\end{cases}\textrm{}\label{eq:qd}
\end{equation}
where $x_{t}$ is time periodic and related to the energy of the quantum
dot, the chemical potential and the temperature of the bath.

\item
In \cite{Schuller} we have a single defect center in natural IIa-type
diamond excited by a red and a green laser with time periodic
intensity. The corresponding rates are
\begin{equation}
\begin{cases}
r_{t}(0,1)=a_{0}\bigl(1+\gamma\sin\bigl(\frac{2\pi}{T_{0}}t\bigr)\bigr)\,,\\
r_{t}(1,0)=b_{0}\,.
\end{cases}\textrm{}\label{eq:sdc}
\end{equation}

\item
In \cite{McNamara} we have a two state model of stochastic resonance given by
\begin{equation}
\begin{cases}
r_{t}(0,1)=\exp\bigl(-k\cos\bigl(\frac{2\pi}{T_{0}}t\bigr)\bigr)\,,\\
r_{t}(1,0)=\exp\bigl(k\cos\bigl(\frac{2\pi}{T_{0}}t\bigr)\bigr)\,.
\end{cases}\label{eq:sr-1}
\end{equation}

\item
In \cite{Verley} it is discussed a piecewise constant and symmetric protocol
\begin{equation}
\begin{cases}
r_{t}(0,1)=\exp\bigl(-h_{t}\bigr)\,,\\
r_{t}(1,0)=\exp\bigl(h_{t}\bigr)\,,
\end{cases}\textrm{ with }h_{t}=\begin{cases}
h_{0}-a & \textrm{if  }0\leq t\leq\alpha T_{0}\,,\\
h_{0}+a & \textrm{if }\alpha T_0\leq t\leq T_{0}\,,
\end{cases}\label{eq:gat-1}
\end{equation}
for some $\a \in (0,1)$.
\end{itemize}

Let us now discuss some results concerning the general situation.
In all this section we restrict to elements
\blu{$\mu,Q,J$ with $(\mu,Q)\in \L$ and $(\mu, J)\in \L_a$,   without further mention.}
For convenience  we call $\mu_t:=\mu_t(0)$, $Q_t:=Q_t(0,1)$ and
$J_t:=J_t(0,1)$ (note that this is different from the usual notation);
accordingly, the jump rates are here denoted by $r_t(0,1)$ and $r_t(1,0)$.
The continuity equation is simply $\partial_t\mu_t+J_t=0$.
\blu{Note that,  by the above continuity equation,   the knowledge of $\mu_t$ and $Q_t$ allows to recover
 \begin{equation}\label{niccolo}
 \mu_t(1)= 1-\mu_t(0)\,, \qquad Q_t(1,0)= \partial_t\mu_t + Q_t\,,\qquad J_t=-\partial_t \mu_t\,.
 \end{equation}
  On the other hand, given real functions $\mu_t$ and $Q_t$ defined for $t \in \cS_{T_0}$ and setting \eqref{niccolo},  we have that  $(\mu, Q)\in \L$ if  and only if  $\mu_t \in [0,1]$, $Q_t\geq 0$  and $\partial_t \mu_t+Q_t\geq 0$.  }

\blu{Moreover (recall that we restrict to   $(\mu,Q)\in \L$)}
the LD  rate functional of Theorem \ref{teo2}  becomes
  \begin{equation}\label{spin}
\begin{split}
I(\mu,Q)=  \int_0^{T_0}\Big[Q_t\log\frac{Q_t}{\mu_t r_t(0,1)} & +(\partial_t \mu_t+Q_t)\log\frac{\left(\partial_t \mu_t+Q_t\right)}{(1-\mu_t)r_t(1,0)} \\
  & +\mu_t r_t(0,1)+(1-\mu_t)r_t(1,0)-2Q_t\Big]\,dt\,.
\end{split}
\end{equation}


\medskip
In this case, one can compute explicitly the LD rate functional
$\blu{I_{m}}(\mu)=\inf_{Q}I(\mu,Q) $ associated to the extended empirical
measure $\mu^{(n)}$.
We have that
$\blu{I_m}(\mu)$ coincides in this case with the joint LD functional  for measure and current, i.e
$\blu{I_m}(\mu)=\hat I(\mu,J)$.  This is because
the current is completely determined by the density using $\partial_t\mu_t=-J_t$ (this fact is indeed true  for more general Markov chains, indeed it is enough that  the \rosso{undirected} graph obtained from the transition graph by disregarding the orientation \rosso{ and identifying multiple edges} is a tree).   The rate functional $\blu{I_m}(\mu)$ is  therefore obtained
as
\begin{equation}\label{lenticchie}
\blu{I_m}(\mu)=I(\mu,Q(\mu,\partial_t \mu))\,,
\end{equation} where (cf. \eqref{domus})
\begin{equation}\label{mercato}
Q_t(\mu,\partial_t\mu):=\frac{-\partial_t\mu_t+\sqrt{\left(\partial_t\mu_t\right)^2+4\mu_t(1-\mu_t)r_t(0,1)r_t(1,0)}}{2}\,.
\end{equation}

\rosso{We point out that, in general, given $(\mu,Q) \in \L$  the  level 2.5   rate functional  $I(\mu,Q)$ is   the  time integration of $I_t(\mu_t, Q_t)$ and the latter is related  to the  level 2.5   rate functional (for the non-extended empirical measure and flow) with  frozen jump rates $r(\cdot, \cdot;t)$. One could wonder if the same property holds for the  level 2    rate functional $\blu{I_m}(\mu)$. In particular, for  2-state Markov chains, one could wonder if $\blu{I_m}(\mu)$ equals }
\begin{equation}\label{chch}
\int_0^{T_0} I_t^{\textrm{frozen}}(\mu_t)\, dt=\int_0^{T_0}\left(\sqrt{\mu_{t}r_t\left(0,1\right)}-\sqrt{\left(1-\mu_{t}\right)r_{t}\left(1,0\right)}\right)^2\,dt\,.
\end{equation}
Formula \eqref{chch}  follows by the explicit form of the level 2
rate functional  for a 2-state chain, which is always reversible \cite{DV4}.
\rosso{This property (i.e. the identity between  $I_m(\mu)$ and  \eqref{chch}) does not hold in general.
Indeed,  since  by \eqref{mercato}
\begin{equation}
Q_t(\mu,0)= \sqrt{\mu_t(1-\mu_t)r_t(0,1)r_t(1,0)}
\end{equation}
 it holds
\begin{equation}\label{popov}
 I_t^{\textrm{frozen}}(\mu_t)= I_t ( \mu_t, Q_t(\mu,0) ) \,,
 \end{equation}
 where $I_t(\cdot, \cdot)$ denotes the integrand in the r.h.s. of \eqref{spin}.
 The above identity  \eqref{popov}  and \eqref{lenticchie} imply  that   $I_m (\mu)$ equals \eqref{chch} when  $\mu$ is constant in time (even with time-dependent rates), and implies that   the zero-th order term  of the formal expansion in $\partial_t \mu$ of $I_m(\mu)=I(\mu,Q(\mu,\partial_t \mu))$  coincides with \eqref{chch}.
}  \smallskip

  In \cite{Verley}  the LD rate  functional of the excess entropy flow  (called there ``cumulated work'')  for a two state model with a time symmetric piecewise constant protocol  is computed explicitly (cf.~Equation (20) there). This explicit level 1 LD rate functional  \blu{could} be obtained by the contraction from  our previous formulas.
\smallskip

The 2-state case is simple enough to allow also an explicit
computation of the non-equilibrium oscillatory state $\pi$. By a
direct computation we have \begin{align*}
& \pi_t(0)=\frac{e^{-\Gamma_t}}{1-e^{-\Gamma_{T_0}}}\left[\int_0^tr_s(1,0)e^{\Gamma_s}\,ds+
e^{-\Gamma_{T_0}}\int_t^{T_0}r_s(1,0)e^{\Gamma_s}\,ds\right]\,,\\
&\blu{  \pi_t(1)=\frac{e^{-\Gamma_t}}{1-e^{-\Gamma_{T_0}}}\left[\int_0^tr_s(0,1)e^{\Gamma_s}\,ds+
e^{-\Gamma_{T_0}}\int_t^{T_0}r_s(0,1)e^{\Gamma_s}\,ds\right]\,,}
\end{align*}
where $\Gamma_t:=\int_0^t\left[r_s(0,1)+r_s(1,0)\right]\,ds$. \blu{Indeed,  it is simple to verify that
$\pi_t(0)\geq 0 $, $\pi_t(1)\geq 0$, $\pi_t(0)+\pi_t(1) =1$ and that  the continuity equation, which reduces to
$\partial_t \pi_t(0)+ \pi_t(0) r_t(0,1)-\pi_t (1) r_1 (1,0)=0$, is fulfilled.  Note
that \cite[Prop. 3.13]{FGR} provides an alternative formula for $\pi_t$.} Recall that $I(\mu,Q)$ is zero when $\mu_t (y)=\pi_t(y)$ and $Q_t(y,z)= \pi_t (y) r_t(y,z)$.

\smallskip
From now on we restrict to the special case
$r_t:=r_t(0,1)=r_t(1,0)$. In this case it is possible to obtain an
explicit expression for the rate functional \blu{$\bar I_{f} \left(\bar
 Q\right)$} of the empirical flow $\bar Q_T$ when $T\to +\infty$ (see
Remark \ref{parto}).
By the graphical construction,  since the jump rates  are the same, we have that
$\bar Q_T$ coincides up to negligible terms with $\frac{\mathcal N_T}{2T}$ where $\mathcal N_T$
is a non-homogeneous Poisson process with periodic intensity  given by  $r_t$. When $T=nT_0$
we can write $\mathcal N_T=\sum_{i=1}^nY_i$, where the $Y_i$ are i.i.d Poisson
random variables of parameter $\int_0^{T_0}r_tdt$. The variable $Y_i$ represents the
number of points in the interval $((i-1)T_0,iT_0]$.   Using the
classic Cram\'er's theorem we deduce  that
\begin{equation}\label{ididimarzo}
\blu{\bar I_f}(\bar Q)=2\bar Q\log\left[\frac{2\bar Q}{\bar r }\right]-2\bar Q+\bar r \,, \qquad \bar r:=\frac{1}{T_0}\int_0^{T_0}r_tdt\,.
\end{equation}
The above result can be also obtained variationally by showing that
the minimizer in
\begin{equation}\label{mqbar}
\blu{\bar I_f} \left(\bar Q\right):=\frac{1}{T_0}\inf_{\left\{(\mu,Q):\frac{1}{T_0}\int_0^{T_0}Q_tdt=\bar Q\right\}}I(\mu,Q)\,,
\end{equation}
is given by $\mu_t=\frac 12$ and $Q_t=r_t\bar Q /\bar r $. We omit the
computations.

\subsection*{Comparison with an effective time homogenous chain}
Always in the case of equal jump rates, \blu{i.e. $r_t:=r_t(0,1)=r_t(1,0)$}, we here obtain an upper bound
for the rate functional \rosso{ $\bar I (\bar \mu, \bar Q)$ (see \eqref{aiuto}) } in terms of the level 2.5 rate
functional of a time homogenous Markov chain with suitable rates \rosso{(in the same spirit of homogenization theory)}.

Let us call $I^{\bar r}$ the LD  rate functional for the
empirical measure and flow of a 2-state Markov chain having time
independent rates equal to $r(0,1)=r(1,0)=\bar r $. According to
\cite{BFG1,BFG2} we have
\begin{equation}\label{davide49}
I^{\bar r}(\bar \mu,\bar Q)=\bar Q\log\left[\frac{ \bar Q^2}{\bar\mu(0)\bar\mu(1)\bar r ^2}\right]-2\bar Q+\bar r\,,
\end{equation}
where, by the divergence free condition, $\bar Q:=\bar Q(0,1)=\bar Q(1,0)$ \blu{(also below, we restrict to divergence-free flows $\bar Q$, otherwise we have $I^{\bar r}(\bar \mu,\bar Q)=\infty$)}. By minimizing  \eqref{davide49} among $\bar \mu$ and comparing with \eqref{ididimarzo} we get that
$$
\inf_{\bar\mu}I^{\bar r}(\bar \mu,\bar Q)=\blu{\bar I_f} (\bar Q)=\inf_{\bar\mu}\bar I(\bar \mu,\bar Q)\,.
$$
In addition  we can  show the inequality   \begin{equation}\label{basket}
\bar I(\bar \mu,\bar Q)\leq I^{\bar r}(\bar \mu,\bar Q)\,,
\end{equation}
which in general is strict.
Inequality \eqref{basket} can be derived  simply by  inserting in \eqref{aiuto}   the special pair $(\mu,Q)$ given by
$$
\mu_t(y)=\bar \mu(y) \,,\qquad Q_t(y,z)=\frac{r_t(y,z) \bar Q(y,z)}
{\bar r }\,.
$$
Considering more general Markov chains one cannot expect inequality
\eqref{basket} to be true.  Indeed, such an inequality would imply
that the rate functionals have the same global minima, which in
general is not valid, see Remark \ref{parto2}.

\section{Preliminary results}\label{sec_preliminare}
In this section we collect some technical results. Since some of them  will be applied also to a tilted   continuous time Markov chain with less regular jump rates, here  we only assume that the jump rates satisfy the periodicity assumption (i.e.  $r(\cdot,\cdot;t)= r(\cdot,\cdot; t+T_0)$ for some  $T_0>0$), assumptions (A1) and  (A2) and that  $r(y,z; \cdot)$ is a measurable,  locally integrable nonnegative function. As mentioned in Section \ref{MR}, the last assumption guarantees that  the associated continuous time Markov chain is well defined \cite{D}.

\begin{definition}\label{defi} Given $\mu \in  \mc M_{+,T_0} (V \times \cS_{T_0})$
 we define $Q^\mu  \in \mc M_+ (E \times \cS_{T_0})$ as $Q^\mu  (y,z,dt ) := \mu(y,dt ) r (y,z;t )$. If $\mu= \mu_t dt $, then
 we set $Q^\mu_t  (y,z ):=\mu_t(y) r (y,z;t )$ (thus implying that $Q^\mu= Q^\mu_t dt $).
\end{definition}

\subsection{Radon-Nikodym derivative}\label{sec_RN}
 Calling $N_t$ the number of jumps of  the trajectory $X$ up to time  $t$, and $\t_1<\t_2 < \cdots <\rosso{\t_{N_t}}$ the jump times, then it holds for $0<t_1<t_2<\cdots < t_n<\rosso{t}  $ \blu{and $x_1, x_2, \dots, x_n\in V$}
 \begin{equation}\label{diderot}
 \begin{split}
 \bbP_{x}  & \left(
 \begin{array}{c}
  N_t=n\,, \t_i \in [t_i, t_i +dt_i)   \text{ for all } i=1,2, \dots, n \,,\\
\blu{ \xi_t=x_i \text{ for all } t\in [t_i,t_{i+1})\,, \forall i =0,1, \dots, n}
\end{array}
   \right)\\ =
 &\exp\Big\{-\sum_{i=0}^n  \int _{t_i} ^{t_{i+1} }  r (x_i;s) ds  \Big\} \prod _{i=0}^{n-1} r (x_i,x_{i+1}; t_{i+1}) dt_1 dt_2\dots dt_n\,,
\end{split}
\end{equation}
 where $x_0:= x$, $t_0:=0$  and $t_{n+1}:= t$. \rosso{We recall that $r(x;s):= \sum _z r(x,z;s)$.}

We consider another Markov chain on $V$ with $T_0$-periodic rates $\bar r(y,z;t)$  (given by nonnegative locally integrable functions)
 and such that
 \[  \bar r( y,z; t)>0 \qquad \Longrightarrow \qquad  r(y,z;t) >0 \,.\] Then its law $\bar \bbP_x  |_{[0,t]}$ on the  space $D([0,t]; V)$  of c\`adl\`ag paths is absolutely continuous
 with respect to $\bbP_x  |_{[0,t]} $ and
 the Radon-Nikodym derivative  on $D([0,t]; V)$   is given by
 \begin{equation}\label{RN}  \frac{ d \bar \bbP_{x}}{d \bbP_{x} }\Big|_{[0,t]}  \bigl( (X_s)_{s \in [0,t]} \bigr) = \exp \left\{
  \int_0^t \bigl[ r(X_s; s)- \bar r  ( X_s; s)\bigr] ds
 \right\} \prod _{\substack{s \in (0,t]:\\ X_{s-} \not = X_s}} \frac{\bar r   (X_{s-} , X_s; s)}{   r (X_{s-} , X_s;s) }\,.
\end{equation}

%
%
%
%

Let us suppose that $ r(y,z;t) =0 $ if and only if $\bar r(y,z;t)=0$. Then we can write
\[ \bar r (y,z;t)= r(y,z;t) e^{F(y,z;t)}\,, \qquad F(y,z;t):= \log \frac{ \bar r (y,z;t) }{r(y,z;t)}  \]
(above we used the convention $\log (0/0)=0$). Note that $F$ is \blu{$T_0$--periodic}.
Since $r(y; \cdot)$ and $r(y,z; \cdot)$ are $T_0$-periodic functions, we can restate \eqref{RN} as follows:
 \begin{equation}\label{RN-new}
 \frac{ d  \bbP^F _{x}}{d \bbP_{x} }\Big|_{[0,nT_0]}   = \exp \Big\{
 n \mu^{(n) } ( r-\bar r) + n  Q^{(n)} ( F) \Big\}\,,\qquad \bbP^F_x:=\bar \bbP_x\,.
 \end{equation}



\subsection{Some identities}\label{rocco}

 Take $Q \in \cM_+ ( E \times \cS_{T_0})$. Denoting by $\cB$ the Borel sets of $\cS_{T_0}$,  for each $y\in V$
 \[
  \cB\ni A \mapsto  \sum_z Q( y,z, A) - \sum _{\rosso{z}} Q(z,y,A)=: \div Q(y, A)\in \bbR
  \]
  is a signed measure on $\cS_{T_0}$. In what follows we denote by $ \div Q( f)$ the integral of $f$ w.r.t. the above measure $\div Q$:
    \begin{equation}
   \div Q( f)= \sum _y \int_0^{T_0} \div Q(y, ds) f(y,s)\,,\qquad f : V \times \cS_{T_0}\to \bbR\,.
  \end{equation}

 \medskip

 \begin{lemma}\label{salmone}
 Let $ f : V \times \cS_{T_0}\to \bbR
$ be $C^1$.   Then
 \begin{equation}\label{fenomeno4}
 \mu^{(n)} (\partial_s f)  - \div Q^{(n)}( f)=\frac{1}{n} \bigl( f( X_{nT_0},0)- f( X_0,0) \bigr)\,.
 \end{equation}
 \end{lemma}
 \begin{proof}
Let $s_1<s_2 <\cdots <s_m$ \rosso{be} the jump times of the path $X$ in the time interval $(kT_0, ( k+1)T_0]$. We set $s_0:= kT_0$  and $s_{m+1}:=( k+1)T_0$.
We can write
\begin{equation*}
\begin{split}
 &  f( X_{(k+1) T_0}, (k+1)T_0) - f( X_{k T_0}, kT_0) =
 f ( X_{s_m}, ( k+1)T_0) - f (X_{s_1 -},kT_0) \\
 & = \sum _{j=0}^{m} [
  f ( X_{s_j}, s_{j+1} ) - f (X_{s_j},s_j)
] +  \sum _{j=1}^{m} [
  f ( X_{s_j}, s_j ) - f (X_{s_j -},s_j)
]
 \\
 & = \sum _{j=0}^{m} \int_{s_j}^{ s_{j+1}  } \partial_s f ( X_s, s) ds
+  \sum _{j=1}^{m} [
  f ( X_{s_j}, s_j ) - f (X_{s_j -},s_j)
] \\
&= \int_{kT_0} ^{(k+1)T_0}   \partial_s f ( X_s, s) ds
 +  \sum _{j=1}^{m} [
  f ( X_{s_j}, s_j ) - f (X_{s_j -},s_j)
]\,.
\end{split}
\end{equation*}
Averaging the above identities among $k=0,\dots, n-1$ and using the  $T_0$-periodicity of $f$  we get
\begin{equation*}
\begin{split}
& \frac{1}{n} \bigl( f( X_{nT_0},0)- f( X_0,0) \bigr)
=
\mu^{(n)}(\partial _s f)  +\sum_{y,z} \int_{[0,T_0] } Q^{(n) } (y,z, ds ) \bigl( f(z,s)-f(y,s) \bigr)\\
& =
\mu^{(n)}(\partial _s f)  - \sum_{y} \sum_z  \int_{[0,T_0] } Q^{(n) } (y,z, ds ) f(y,s)  +\sum_y \sum_z  \int_{[0,T_0] } Q^{(n) } (z,y, ds )f(y,s)\\
& = \mu^{(n)}(\partial _s f)  - \div Q^{(n) }(f)\,.
\end{split}
\end{equation*}
\end{proof}


\subsection{The oscillatory steady state}
We collect in the following proposition some asymptotic  properties of  the oscillatory steady state.  Recall \blu{the definition of $\pi$ given in Section \ref{MR} and} Definition \ref{defi}.
\begin{proposition}\label{lln} The following holds:
\begin{itemize}
\item[(i)]
Fixed $t\in [0,T_0]$, under $\bbP_x$,  the law of $X_{t+nT_0}$ weakly converges to $\pi_t$ as $n$ goes to $\infty$;
\item[(ii)] $\bbP_x$-a.s.  $\mu^{(n)}$ weakly converges to   \blu{$\pi=\pi_t dt $}  in  $\mc M_{+,T_0}( V \times \cS_{T_0} )$. \blu{More generally, given a measurable function $f: V \times \cS_{T_0} \to \bbR$  with $\|f\|_\infty<\infty$, it holds
\begin{equation}\label{mazurka1}
\lim_{n \to \infty } \mu^{(n)} (f) =  \pi(f) \qquad \text{$\bbP_x$-a.s.  and in $L^1(\bbP_x)$\,.}
\end{equation}}
\item[(iii)]
$\bbP_x$-a.s.,
$\bar \mu_T$ weakly converges to $\frac{1}{T_0} \int _0 ^{T_0} \pi_t dt$ in $\cP(V)$;
\item[(iv)]  $\bbP_x$-a.s.  $Q^{(n)}(y,z ; dt)$  \rosso{weakly} converges to   $Q_t^\pi  (y,z)  dt$  in $\mc M_+ (E \times \cS_{T_0})$. \blu{More generally, given a measurable function $g: E \times \cS_{T_0} \to \bbR$  with $\|g\|_\infty<\infty$, it holds
\begin{equation}\label{mazurka2}
\lim_{n \to \infty } Q^{(n)} (g) =  Q^\pi (g)  \qquad \text{$\bbP_x$-a.s.  and in $L^1(\bbP_x)$\,.}
\end{equation}}
\end{itemize}
\end{proposition}
\begin{proof}
(i) Due to Assumptions (A1) and (A2), the discrete time Markov chain  $(X_{t+nT_0})_{n \geq 0}$
is irreducible. Since $V$ is finite, we get that this discrete time Markov chain has a unique invariant distribution to which it converges (whatever the initial distribution). As a consequence, the invariant distribution must be given by the distribution $\pi_t$ introduced in Section \ref{MR}. This concludes the proof of Item (i).
\rosso{Item (iii) follows directly  from Item (ii). The proof of Items (ii) and  (iv) can be derived from  \cite[\rosso{Thm.} 2.1]{HK} adapted to processes  with c\`adl\`ag paths.  We comment this step.  We associate to the continuous time  Markov chain  $\xi = (\xi_t)_{t\in \bbR_+}$  the random sequence  $\bbX=(\bbX_k)_{k \geq 0 }$   of paths  in $D([0,T_0]; V)$ with $\bbX_k:= (\xi_{k T_0  +s  })_{0 \leq s \leq T_0}$.   By the arguments presented to derive  Theorem 2.1 in \cite{HK} we get that $\bbX$ is a Markov chain, ergodic and stationary  when $\xi_0$ is sampled with distribution $\pi_0$. Hence $\bbP_{\pi_0}|_{[0,T_0]}$ is the marginal distribution of $\bbX$ in the stationary state. Given  \blu{measurable} functions $f: V \times \cS_{T_0} \to \bbR$ and $g: E \times \cS_{T_0}\to \bbR$ \blu{with $\|f\|_\infty, \|g\|_\infty< \infty$}, we can write (see \eqref{roma_as} and \eqref{juve})
\begin{align}
& \mu^{(n)} (f) =\frac{1}{n} \sum _{j=0}^{n-1} F(\bbX_j)\,,\qquad F(\z) := \int_0^{T_0}   f(\z_t,t)dt \\
& Q^{(n)}(g) = \frac{1}{n}  \sum _{j=0}^{n-1} G(\bbX_j)+O(1/n)   \,, \qquad G(\z) := \sum _{\substack{  t\in [0,T_0) :\\ \z_{t-}\not =\z_t } }  g( \z_{t-}, \z_t ,t) \,,
\end{align}
where $\z$ denotes a   generic element of $D([0,T_0]; V)$.}
\blu{We observe that  $F, G$ are integrable  w.r.t. $\bbP_{\pi_0}|_{[0,T_0]}$. This is trivial for $F$ since bounded.
 The integrability of  $G$ follows from the boundedness of $g$ and the fact that the total  number of  jumps in $[0,T_0)$ under $\bbP_{\pi_0}|_{[0,T_0]}$ is stochastically dominated by a suitable Poisson random variable due to Assumption (A3), hence  $G\in L^1 (\bbP_{\pi_0})$.
 From  Birkhoff's ergodic theorem  we derive the $ \mu^{(n)} (f) $ converges to $\bbE_{\pi_0} [ F]$ and $Q^{(n)}(g) $ converges to $\bbE_{\pi_0} [ G]$  both  $\bbP_{\pi_0}$--a.s. and in $L^1 (\bbP_{\pi_0})$. Since $\bbP_{\pi_0}= \sum \pi_0(x) \bbP_x $ and $\pi_0(x)>0$ for any $x$, we derive the convergence also $\bbP_x$--a.s. and in $L^1 (\bbP_x)$ for any $x \in V$.}
\end{proof}

\begin{lemma}\label{conti}
It holds
$
\partial _t \pi_t +\div Q^\pi _t=0$ \blu{weakly}.
\end{lemma}
\begin{proof}
Due to Definition \ref{defi}  we only need to prove that $\rosso{\pi} (\partial_s f)  -  \div Q^\pi( f)=0 $ for any $C^1$ function
\rosso{$ f : V \times \cS_{T_0}\to \bbR
$}. This identity can be obtained by taking the limit $n\to \infty$ in
 Lemma \ref{salmone} and using Proposition \ref{lln}.
 \end{proof}

We conclude this section with an alternative characterization of
$\pi=\pi_t dt $.
\begin{proposition}\label{ananas}
 The only weak solution $\mu\in   \mc M_{+,T_0} (V \times \cS_{T_0}) $, with  $\mu=\mu_t dt $, of the
 equation
 \begin{equation}\label{arancia}
\partial _t \mu_t + \div Q^\mu_t =0
\end{equation}
is given by $\pi$.
\end{proposition}
\begin{proof}
We first show that $\mu\in   \mc M_{+,T_0} (V \times \cS_{T_0}) $, with  $\mu=\mu_t dt $, solving \eqref{arancia} is  an invariant measure of
 the piecewise deterministic Markov process \blu{$(W_t,Y_t)_{t\geq 0}$} on $V\times \cS_{T_0}$ \blu{with extended generator \cite{D} given by \eqref{ago}} (we  are making  some  slight abuse of notation, since $(x,s)$ in \eqref{ago} has to be thought of as element of $V\times \cS_{T_0}$ via the canonical projection for times). By \cite[\rosso{Thm.} (26.14)]{D}  the domain of the extended generator is given by the functions $f(x,s)$ which are absolutely continuous in $s$ (we shortly write $f\in \cA\cC$).
Hence,  due to \cite[\rosso{Prop. (34.7)}]{D}, $\mu$ is an invariant measure for the PDMP  if and only in  $\mu ( Lf  )=0$ for any $f\in \cA\cC$. By density, it is enough that
$\mu ( Lf  )=0$ for any $C^1$ function  $f$, which (by   integration by parts)
is equivalent to the fact that $\mu$ is a weak solution of \eqref{arancia}.

\blu{We set  $p_{s,s+t}(x,y):= P( \xi_{t+s}=y|\xi_s=x)$.
Since $\mu$ is an invariant measure for the PDMP, given a $C^1$ function $f$ on $V \times \cS_{T_0}$,  it holds
\begin{equation}\label{legoland}
\begin{split}
\sum _x \int _{\cS_{T_0}}  \mu_s(x) f(x,s) ds & = \sum_x \int _{\cS_{T_0}} \mu_s(x) E \bigl[ f(W_{T_0}, Y_{T_0})| W_0=x, Y_0=s\bigr]ds\\
& =\sum_x \sum_y \int _{\cS_{T_0}} \mu_s(x)  p_{s, s+T_0}(x,y) f( y,s+T_0)ds\\
& = \sum_y \sum _x \int _{\cS_{T_0}} \mu_s(y)  p_{s, s+T_0}(y,x)  f( x,s)ds\,.
\end{split}
\end{equation}
Note that,  in the last identity, we have used the $T_0$--periodicity of $f$.
By the  $T_0$--periodicity of the map $s\mapsto \mu_s$  and by  the arbitrariness of $f$ in \eqref{legoland}, we conclude that
 $\mu_{s+T_0}(x)=\mu_s(x)= \sum_y \mu_s(y)  p_{s, s+T_0}(y,x)$. This is the equation characterizing $\pi$, apart a multiplicative factor. As a consequence we get that $\mu= c \pi$ for some factor $c$. Since both $\mu$ and $\pi$ have total mass $T_0$, we conclude that $\mu=\pi$.
 On the other hand, it is simple to check (by the arguments presented above) that $\mu:=\pi$ solves \eqref{arancia}.
 }
 \end{proof}


\section{Proof of Theorem \ref{teo2}: upper bound \eqref{eq_UB}, convexity and \blu{lower-semicontinuity} of $I$}\label{sec_UB}

We start by showing exponential tightness:
\begin{lemma}\label{teso}
The family $\left	\{ \bb P_x \circ (\mu^{(n)},Q^{(n)})^{-1}\right\}_{n \geq 1}$ of probability measures on $\cM_*$    is exponentially tight.
\end{lemma}
\begin{proof}
Given $\ell >0$ we set  $\cK_\ell:=\{ (\mu,Q) \in \cM_* \,:\, Q(1)\leq \ell\}$.  \rosso{Above, $Q(1)$ denotes the averaged value  w.r.t. to the measure $Q$ of the function constantly equal to $1$, equivalently $Q(1)$ is the total mass of the measure $Q$.}
 Then $\cK_\ell$ is a compact subset of $\cM_*$  \cite{Bi}. To prove the exponential tightness  it is enough to show  that there exists $C>0$ such that
\begin{equation}\label{doraemon}
\varlimsup _{n \to \infty} \frac{1}{n}\log  \bbP_x( \, (\mu^{(n)},Q^{(n)}) \not \in \cK_\ell \,) \leq -C \ell
\end{equation}
for large $\ell$.

We prove \eqref{doraemon}. The event $\{(\mu^{(n)},Q^{(n)}) \not \in \cK_\ell \}$ is simply the event that  the measure $Q^{(n)}$ has total mass larger than $\ell$. Due to \rosso{\eqref{juve}}, the total mass of $Q^{(n)}$ equals $1/n$ times the number of jumps in the time interval $[0,nT_0]$. On the other hand, by the graphical construction presented in Section
 \ref{grafica} the number of jumps in the time interval $[0,nT_0]$ is stochastically dominated by a Poisson variable $Z$ of parameter $\l n T_0$ where $ \l= \sum _{(y,z)} \sup _{ t\in [0,T_0]} r(y,z;t) $. Since $E[ e^{\g Z}]= \exp\{ \l n T_0  (e^{\g }-1)\}   $, by applying Chebyshev's inequality  we get
 \begin{equation}
 \begin{split}
  \bbP_x( \, (\mu^{(n)},Q^{(n)}) \not \in \cK_\ell \,) &  =  \bbP_x( \, Q^{(n)}(1)> \ell \,) \leq P( Z >n  \ell) \\& \leq
e^{-n \ell}E[ e^Z]= \exp \{ -n\ell + \l n T_0 (e-1) \}\,.
\end{split}
\end{equation}
The above bound trivially implies \eqref{doraemon}.
\end{proof}



\rosso{Recall that $r (y;t)= \sum_z r(y,z;t) $. Given a continuous function  $F: E \times \cS_{T_0} \to \bbR$, we set
\[
 r^F (y,z;t)= r(y,z;t) e^{F(y,z;t)} \qquad \text{ and } \qquad  r^F (y;t)= \sum_z r^F(y,z;t)\,.
 \]
Moreover, we consider
$\phi: V \times \cS_{T_0} \to \bbR $  of class  $C^1$ and we define the mappings  $\hat{I}_{\phi,  F}:\cM_* \to \bbR_+$ and $ I_{\phi,  F} :\cM_*\to [0,+\infty] $ as follows:
}
 \begin{align}
& \hat{I}_{\phi,  F}(\mu, Q):= - \mu ( \partial _t \phi) + \div Q (\phi)+    Q(F) -\mu( r^F-r)\,, \label{africa1}\\
&   I _{\phi,  F}(\mu, Q):=
     \begin{cases}
       \hat{I}_{\phi,  F}(\mu, Q)  & \text{ if } \mu = \mu_t  dt\,, \; \mu_t(V)=1 \text{ a.s. } \\
  +\infty & \text{ otherwise}\,.
 \end{cases} \label{africa2}
 \end{align}
 \begin{lemma}\label{fuji}
 The function $I_{\phi,F}$ is convex and lower semicontinuous.
 \end{lemma}
 \begin{proof}
 Let us call $\cA$  the set of  pairs $(\mu,Q)\in \cM_*$ such that  $\mu = \mu_t  dt$, $\mu_t(V)=1$ \blu{for almost all $t\in \cS_{T_0}$}. It is simple to check that $\cA$ is convex and closed in $\cM_*$. Since $\cA$ is convex and $\hat{I}_{\phi,  F}$ is  convex, it is simple to derive  that $I _{\phi,  F}(\mu, Q)$ is convex.
%

 Let us now prove that   $I _{\phi,  F}$ is continuous on $\cA$. To this aim, given   $(\nu^{(k)}, Q^{(k)})  \to (\nu,Q)$  in $\cA$, we need to show that   $\hat{I}_{\phi,  F}( \nu^{(k)},Q^{(k)})\to\hat{I}_{\phi,  F}( \nu,Q)$.
  Due to the definition of weak convergence of measures and since $\partial_t\phi$, $\phi$ and $F$ are continuous,   the only non trivial step is to show that $\nu^{(k)} (h)\to \nu  (h)$ where $h:=r^F-r$. Since $h(y;t) = \sum _z r(y,z ;t) [ e^{ F(y,z;t)}-1]$, and $F$ is \blu{continuous} in time, for each $y $ the function $h(y; \cdot)$ is continuous on $ \cS_{T_0}\setminus \cD$ (recall Assumption (A4)). On the other hand, since $ \nu =\nu_t dt$, we have $ \sum _y \nu (  y, \cD)=0$. As a \rosso{ byproduct } of the last observation and  the Portmanteau theorem as stated in \cite[Thm. 12.6]{MP}, we  get that  $\nu^{(k)} (h)\to \nu  (h)$.
This concludes the proof that  $I _{\phi,  F}$ is continuous on the  set $\cA$.
    Since  $I _{\phi,  F}$ is continuous on the closed set $\cA$ and it equals $+\infty$ on $\cM_*\setminus \cA$, we conclude that   $I _{\phi,  F}$  is lower semicontinuous.
 \end{proof}
%
%
%
\bigskip

Let us define
\begin{equation}
M_n ^F:=  \exp \Big\{-
 n \mu^{(n) } ( r^F- r) + n  Q^{(n)} ( F) \Big\}\,.
 \end{equation}
 We recall \rosso{ that by \eqref{RN-new}}
\[
 \frac{ d  \bbP^F_{x}}{d \bbP_{x} }\Big|_{[0,nT_0]}   = M_n^F
\]
where $ \bbP^F_{x}$ is the law of the new Markov chain with jump rates $r^F(y,z; t)$.

 Due to \eqref{fenomeno4} we can write
 \begin{equation}\label{caramello}
 -n I_{\phi, F}( \mu^{(n)}, Q^{(n)} ) =
  \phi( X_{nT_0},0)- \phi( X_0,0)- \log  M^F_n\,.
 \end{equation}
 In the above identity we have used also that
 $\mu^{(n)}(x,dt )= \mu^{(n)} _t (x) dt $ where $0\leq \mu^{(n)} _t (x)  \leq 1$ (cf. \eqref{poretti}), thus implying that $ I_{\phi, F}( \mu^{(n)}, Q^{(n)} )
=\hat  I_{\phi, F}( \mu^{(n)}, Q^{(n)} ) $.

\begin{lemma}\label{gatto}
Fix $x \in V$. For each $\phi, F$ as above and each measurable $\cB \subset \cM_*$ it holds
\begin{equation}
\varlimsup_{n \to \infty} \frac{1}{n} \log  \bbP_x \left(\,  (\mu^{(n)}, Q^{(n)}) \in \cB \right) \leq - \inf _{(\mu, Q) \in \cB}  I_{\phi, F} (\mu, Q)\,.
\end{equation}
\end{lemma}
\begin{proof}
Due to \blu{\eqref{caramello}} we can write
\begin{equation*}
\begin{split}
& \bbP_x  \left(\,  (\mu^{(n)}, Q^{(n)}) \in \cB \right) \\
& = \bbE_x \bigl (\exp \bigl\{ - n I_{\phi, F}( \mu^{(n)}, Q^{(n)} )
- [\phi( X_{nT_0},0)- \phi( X_0,0)]\bigr\} M_n^F  \mathds{1} _{\cB} ( \mu^{(n)}, Q^{(n)} )   \bigr)\\
&\leq \bigl[ \sup_{(\mu, Q)\in \cB} e^{- n I_{\phi, F} (\mu,Q)}\bigr] e^{2 \| \phi\|_\infty}\bbE_x (M_n^F)= \bigl[ \sup_{(\mu, Q)\in \cB} e^{- n I_{\phi, F} (\mu,Q)}\bigr] e^{2 \| \phi\|_\infty}\,,
\end{split}
\end{equation*}
thus implying the thesis.
\end{proof}

Due to the exponential tightness \blu{(see Lemma \ref{teso})}, it is enough to prove the upper bound \eqref{eq_UB} for compact subsets $\cK\subset \cM_*$ instead of generic closed subsets $\cC\subset \cM_*$.   Due to Lemma \ref{gatto}, \rosso{for any open subset $\cO\subset \cM_*$} we have
\[
\varlimsup_{n \to \infty} \frac{1}{n} \log \bbP_x \left(\,  (\mu^{(n)}, Q^{(n)}) \in \rosso{\cO} \right) \leq- \sup_{\phi,F}  \inf _{(\mu, Q) \in \rosso{\cO}}  I_{\phi, F} (\mu, Q)\,.\]
As a byproduct of the above bound and the minmax lemma (cf. \cite[Lemma 3.3, App. 2]{KL}), we conclude that
\[
\varlimsup_{n \to \infty} \frac{1}{n} \log \bbP_x \left(\,  (\mu^{(n)}, Q^{(n)}) \in \cK \right) \leq-   \inf _{(\mu, Q) \in \cK}\sup_{\phi,F}  I_{\phi, F} (\mu, Q)\,.\]
Hence, to conclude the proof of the upper bound \eqref{eq_UB} it is enough to apply the following lemma:
\begin{lemma}\label{mauro}
For each $(\mu,Q)\in \cM_*$ it holds
\begin{equation}\label{nocciolina}
I(\mu,Q)= \sup_{\phi,F}  I_{\phi, F} (\mu, Q)\,,
\end{equation}
where the supremum is taken among all  $C^1$ functions  $\phi: V \times \cS_{T_0} \to \bbR $  and  continuous functions    $F: E \times \cS_{T_0} \to \bbR$.
\end{lemma}
\begin{remark}\label{I_convessa}
Note that, by the convexity \blu{and the lower-semicontinuity} of $I_{\phi, F} $, the above lemma implies the convexity \blu{and the lower-semicontinuity} of $I$.
\end{remark}

\begin{proof}
In what follows we write $\ell(\cdot)$ for the Lebesgue measure on $\cS_{T_0}$.

\smallskip

$\bullet$ \emph{Case $(\mu,Q)\not \in \L$}.
We claim  that \eqref{nocciolina} reduces to $+\infty=+\infty$ if $(\mu,Q)\not \in \L$.

\smallskip

From the definition of $ I(\mu,Q)$ and $ I_{\phi, F} (\mu, Q)$  \rosso{ (see   \eqref{pakistano} and \eqref{africa2}) }  one trivially gets that both sides of \eqref{nocciolina} are $+\infty$ if  $(\mu ,Q) \not \in \cA$, where $\cA$ is defined
as in the proof of Lemma \ref{fuji}. It \rosso{is} also trivial to verify that both sides of \eqref{nocciolina} are $+\infty$ if, for  some $C^1$ function $\rosso{\phi} :V \times \cS_{T_0}\to \bbR$, it holds $- \mu ( \partial _t \phi) + \div Q (\phi)\not =0$.
Hence, in what follows we  restrict to the case  $\mu = \mu_t  dt$, $\mu_t(V)=1$ \blu{for almost all $t\in \cS_{T_0}$}, and $ \partial_t \mu + \div Q=0$ (in the weak sense). Since in this case   $I_{\phi, F}(\mu,Q)$ does not depend on $\phi$,  we write simply  $I_{F}(\mu,Q)$.

\smallskip

Suppose now that $Q$ is not of the form $Q_t dt $. Hence there exists  a subset $B$ of $\cS_{T_0}$ with zero Lebesgue measure
such that $Q(y_0,z_0,B)>0$ for some $(y_0,z_0)\in E$.    Since both $\ell(\cdot) $ and $Q(y_0,z_0, \cdot)$   are measures of finite mass, they are  regular. Hence,  by \cite[Thm. 1.1]{Bi}, for any $\e>0$ there exist a closed set $D_\e$ and an open set $G_\e$ such that $D_\e\subset B\subset G_\e$,
 $Q(y_0,z_0, G_\e\setminus D_\e)\leq \e$ and $\ell ( G_\e \setminus D_\e) \leq \e$.   In what follows we take $\e<Q(y_0,z_0, B)/2$, thus implying that $Q(y_0,z_0, D_\e)\geq Q(y_0,z_0, B)/2$. On the other hand, since  $\ell (B)=0$, we get that $\ell(G_\e)\leq \e$.
 By \rosso{Urysohn's lemma} we can find a continuous function $\varphi_\e : \cS_{T_0}\to [0,1]$ such that $\varphi_\e \equiv 1 $ on $D_\e$ and $\varphi_\e\equiv 0$ on $G_\e^c$. We then introduce the continuous test function $F_\e(y,z,t)=\g(\e) \d_{y,y_0}\d _{z,z_0} \varphi _\e(t)$ where the positive parameter $\g(\e)$ will be fixed at the end.
 Then we have
\begin{equation}\label{cento}
\begin{split}
I_{F_\e}(\mu,Q) & = \sum_{(y,z) } \int _0^{T_0}  Q (y,z,dt) F_\e(y,z,t)  -\sum _y  \int_0^{T_0}   \mu_t( y)\bigl( r^{F_\e}(y,t)-r(y,t)\bigr)dt\\
&=\int _0^{T_0}  Q (y_0,z_0 ,dt) F_\e(y_0,z_0,t) - \int_0^{T_0}  \mu_t(\blu{y_0})r (y_0,z_0,t)\bigl( e^{ F_\e(y_0,z_0,t)}-1 \bigr) dt\\
&\geq \g(\e) Q (y_0,z_0 ,D_\e) -  e^{\g(\e)}   \int_0^{T_0}  \mu_t( \blu{y_0})r (y_0,z_0,t) \mathds{1} (t\in G_\e)dt \\
& \geq  \g(\e) Q (y_0,z_0 ,D_\e) -  e^{\g(\e)}    \ell (G_\e) \max_{y,z,t} r(y,z,t) \\
&\geq  \g(\e) Q (y_0,z_0 ,B)/2-  e^{\g(\e)}  \e  \max_{y,z,t} r(y,z,t)   \,.
\end{split}
\end{equation}
Taking $\g(\e):= \log (1/\e)$, we get that $\lim _{\e \downarrow  0} I_{F_\e}(\mu,Q)=+\infty$. Hence,
it holds  $\sup _F I_F(\mu,Q)=+\infty$, while trivially $I(\mu,Q)=+\infty$ since $(\mu,Q)\not \in \L$.

\smallskip
We now focus on property (iv) in Definition \ref{waffel} of $\L$. Let us suppose that there \rosso{exist} $B \subset \cS_{T_0}$ and an edge $(y_0, z_0)$ such that $\ell (B)>0$, $\mu_t (y_0)=0$ for   all $t \in B$ and $Q_t(y_0,z_0)>0$ for   all $t\in B$. We need to prove  that $\sup _F I_F(\mu,Q)=\infty$. As above for any $\e>0$ we fix  a closed set $D_\e$ and an open set $G_\e$ such that $D_\e\subset B\subset G_\e$ and $\ell ( G_\e \setminus D_\e) \leq \e$. Without loss \rosso{of generality}  we take $D_{\e} \subset D_{\e'}$ if $\e> \e'$.
  Since $\ell(B)>0$ we have $\ell ( D_{\e_0} )\geq \ell(B)/2>0$ for  $ \e_0:= \ell(B)/2$. In particular,
\blu{$\int _{D_{\e}} Q_t (y_0,z_0) dt  \geq \int _{D_{\e_0}} Q_t (y_0,z_0) dt  >0$ for any $\e<\e_0$}.
 Hence,
similarly to \eqref{cento},  we get
\[ I_{F_\e}(\mu,Q) \geq  \g(\e)
 \int _{D_{\e_0}} Q_t (y_0,z_0) dt  -   e^{\g(\e)}   \blu{ \ell (G_\e\setminus B)} \max_{y,z,t} r(y,z,t)\,.
\]
Using that $   \ell (G_\e\setminus D_\e)\leq \e$ and taking $\g(\e) := \log (1/\e)$, we conclude that $\lim _{\e \downarrow  0} I_{F_\e}(\mu,Q)=+\infty$, thus proving that $\sup _F I_F(\mu,Q)=\infty$.

\medskip

This concludes the proof of our initial claim.

\medskip

$\bullet$ \emph{Case $(\mu,Q) \in \L$}.  We now assume that $(\mu,Q)\in \L$. Since $\partial _t \phi + \div Q=0$, we have $I_{\phi,F}(\mu,Q)=I_{0,F}(\mu,Q)=: I_F(\mu,Q)$. Hence,  we only need to show that
\begin{equation}\label{uff}
I(\mu,Q)= \sup _{F} I_F(\mu,Q)\,,
\end{equation}
where the supremum is taken among  the   continuous  functions    $F: E \times \cS_{T_0} \to \bbR$ and
\begin{equation}\label{picasso}
\begin{split} I_F(\mu,Q) & =  \sum_{(y,z) } \int _0^{T_0}  Q_t (y,z) F(y,z,t) dt -\sum _y  \int_0^{T_0}   \mu_t( y)\bigl( r^F(y;t)-r(y;t)\bigr)dt\\
&= \sum _{(y,z)} \int _0 ^{T_0} dt \Big[  Q_t (y,z) F(y,z,t)- \mu_t(y) r(y,z;t) ( e ^{F(y,z,t)}-1) \Big]\,.
\end{split}
\end{equation}
Since (cf. \eqref{Phi})  $ \Phi (q,p) =\sup_{v\in \bbR} \{ q v- p (e^v-1) \}$ for any $(q,p) \in \bbR_+\times \bbR_+$,  we can bound from above the integrand in the r.h.s. of \eqref{picasso} by $\Phi\bigl( Q_t(y,z), \mu_t(y) r(y,z;t) \bigr)$, thus implying that
\begin{equation}
I_F(\mu,Q)\leq \rosso{\sum_{(y,z)}}\int_0 ^{T_0}  \Phi \big( Q_t(y,z), \mu_t (y) r(y,z;t)\big)  dt = I(\mu,Q)\,.
\end{equation}
It remains to prove that $I(\mu,Q) \leq \sup _F I_F(\mu,Q)$, $F$ varying among the continuous functions.
Since $(\mu,Q) \in \L$  we have
\[I(\mu, Q) = \sum_{(y,z) \in E}
\int_0^{T_0}    \Phi \big( Q_t(y,z), \mu_t (y) r(y,z;t)\big)dt \,.\]
\rosso{Given} $(y,z)\in E$ and given $\e>0$ we define
\begin{align*}
  A(y,z) & :=\{ t\in \cS_{T_0}\,:\,  Q_t (y,z)=0   \}   \,,\\
B(y,z)  & :=\{ t\in \cS_{T_0}\,: \, \mu_t(y)=0 \text{ and } Q_t (y,z)>0 \}\,,\\
 C(y,z) & := \cS_{T_0} \setminus \bigl( A(y,z) \cup B(y,z) \bigr)   =\{ t\in \cS_{T_0}\,:\,  Q_t (y,z)>0 \text{ and } \mu_t(y)>0  \}   \,,\\
   C_\e(y,z) & :=\{ t\in \cS_{T_0}\,:\,  \e\leq Q_t (y,z) \leq \frac{1}{\e}  \text{ and } \e\leq \mu_t(y)\leq \frac{1}{\e}   \}\,.
\end{align*}
Since $(\mu,Q)\in \L$, we have $\ell \bigl( B(y,z)\bigr)=0$. In particular, by definition of $\Phi$ (cf. \eqref{Phi}),
\begin{equation}\label{pietro1}
\begin{split}
 \int_0^{T_0} & \Phi \big( Q_t(y,z), \mu_t (y) r(y,z;t)\big)dt\\
 &   =
  \int _{A(y,z)} \mu_t(y) r(y,z; t) dt + \int_{C(y,z)}
  \Phi \big( Q_t(y,z), \mu_t (y) r(y,z;t)\big)dt\,.
  \end{split}
\end{equation}
Since  $  \Phi \big( Q_t(y,z), \mu_t (y) r(y,z;t)\big) \in \bbR_+$ on $C(y,z)$ we have
\begin{multline}\label{pietro2}
\int_{C(y,z)}
  \Phi \big( Q_t(y,z), \mu_t (y) r(y,z;t)\big)dt=\lim_{\e \downarrow 0} \int_{C_\e(y,z)}
  \Phi \big( Q_t(y,z), \mu_t (y) r(y,z;t)\big)dt\,.
\end{multline}
We now note that
\begin{equation}\label{pietro3}
\sum_{(y,z)\in E}  \int_{C_\e(y,z)}
  \Phi \big( Q_t(y,z), \mu_t (y) r(y,z;t)\big)dt= I_{ F_\e }(\mu,Q)\,,
  \end{equation}
  where
  \[ F_\e(y,z,t):=
  \begin{cases}
  \log \frac{Q_t(y,z)}{\mu_t(y) r(y,z;t) } & \text{ if } t \in C_\e(y,z) \,,\\
  0 & \text{ otherwise}\,.
  \end{cases}
  \]
Given $M>0$ we now define
  \[ F_{M,\e}(y,z,t):=
  \begin{cases}
  -M &\text{ if }\; t \in A(y,z)\,,\\
  \log \frac{Q_t(y,z)}{\mu_t(y) r(y,z;t) } & \text{ if }\; t \in C_\e(y,z) \,,\\
  0 & \text{ otherwise}\,.
  \end{cases}
  \]
  Note that $F_{M,\e}$ is a bounded measurable function.
Since $ \int _{A(y,z)} \mu_t(y) r(y,z; t) dt= \lim _{M\uparrow  \infty} \int _{A(y,z)}  ( 1-e^{-M} )  \mu_t(y) r(y,z; t) dt$, from \eqref{pietro1}, \eqref{pietro2} and \eqref{pietro3} we conclude that
\[I(\mu, Q) = \sum_{(y,z) \in E}
\int_0^{T_0}    \Phi \big( Q_t(y,z), \mu_t (y) r(y,z;t)\big)dt  = \lim _{\e \downarrow 0} I_{F_{1/\e, \e}}(\mu,Q)\,.
\]
Above, we have used the same notation \rosso{as in} \eqref{picasso}, which remains meaningful for bounded measurable functions.
To have \eqref{uff} it is now enough to approximate  $I_{F_{1/\e, \e}}(\mu,Q)$ by  $I_{F}(\mu,Q)$ with $F$ continuous, for any fixed $\e>0$.
To this aim, we recall that by construction $F_{1/\e, \e}$ is a bounded measurable function. Let $\psi_n$ be a sequence of continuous  mollifiers. Then $G_{n,\e} $ defined as the convolution of  $F_{1/\e, \e}$ with $\psi_n$ is a continuous function with $\| G_{n,\e}\|_\infty \leq \| F_{1/\e, \e}\| _\infty$ and such that $G_{n,\e} \to F_{1/\e, \e}$    Lebesgue almost everywhere. By \eqref{picasso} and dominated convergence we then conclude that $\lim _{n \to \infty} I_{G_{n,\e}  }(\mu,Q)= I_{F_{1/\e, \e} } (\mu,Q)$.
 \end{proof}

\section{Proof of Theorem \ref{teo2}: lower bound  \eqref{eq_LB} and goodness of $I$}\label{sec_LB}
\blu{The goodness of the rate functional follows from the exponential tightness in Lemma \ref{teso} and Lemma 4.1.23 in \cite{DZ}.}
Our strategy to prove the lower bound is based on a relative entropy calculation according to the following general result, where $\Ent(\cdot| \cdot )$ denotes the relative entropy of probability distributions.

\begin{lemma}\label{omm}
  Let $\{P_n\}$ be a sequence of probability measures on a Polish space $\mc X$. Assume that for each $x\in\mc X$
  there exists a sequence of probability measures $\{\tilde{P}_n^x\}$
  weakly convergent to $\delta_x$ and such that
  \begin{equation}
    \label{entb}
    \varlimsup_{n\to\infty} \frac 1n \Ent\big(\tilde{P}_n^x \big| P_n\big)
    \le J(x)
  \end{equation}
  for some $J\colon \mc X\to [0,+\infty]$. Then the sequence $\{P_n\}$
  satisfies the large deviation lower bound with rate functional given
  by $\sce J$, the lower semicontinuous envelope of $J$, i.e.
  \begin{equation*}
    (\sce J) \, (x) := \sup_{U \in\mc N_x} \; \inf_{y\in U} \; J(y)
  \end{equation*}
  where $\mc N_x$ denotes the collection of the open neighborhoods of $x$.
\end{lemma}
This lemma has been originally proven in \cite[Prop.~4.1]{Je}, see also \cite[Prop.~1.2.4]{Ma}.

\smallskip

We first prove the \rosso{inequality} \eqref{entb} for the functional $J$ defined
as follows.  Let $\Lambda_0\subseteq \Lambda$ be the collection of
elements $(\mu,Q)\in \Lambda$ such that there exists  \blu{$\epsilon>0$}  for which  $\blu{\mu_t(x)\geq \epsilon}$ and $ \blu{\e^{-1} \geq Q_t(y,z)\geq \epsilon}$ for all  \rosso{$t\in \cS_{T_0}$, $x\in V$} and $(y,z)\in E$. We define
$$
J(\mu,Q)=\left\{
\begin{array}{ll}
I(\mu,Q) & \textrm{if}\ (\mu,Q)\in \Lambda_0\,,\\
+\infty & \textrm{otherwise}\,.
\end{array}
\right.
$$
Then we finish the proof of the lower bound showing that $(\sce J)=I$.

Given $(\mu, Q)\in \Lambda_0$ we consider a Markov chain $\tilde{\mathbb P}$ having
jump rates defined by
\begin{equation}\label{uffona}
\tilde r(y,z;t):=\frac{Q_t(y,z)}{\mu_t(y)}\,.
\end{equation}
We observe that \blu{$\e \leq r(y,z;t) \leq \e^{-2}$ and that } $\mu_t$ satisfies the continuity equation
\begin{equation}\label{tildeK}
\partial_t\mu _t+\div \tilde Q^\mu_t =0\,.
\end{equation}
The symbol $\tilde Q^\mu$ in \eqref{tildeK} is defined like in Definition \ref{defi} by $\tilde Q^\mu_t(y,z):=\mu_t(y)\tilde r(y,z;t)$. Trivially,  $\tilde Q^\mu=Q$ and therefore \eqref{tildeK} follows from the definition of $\L_0$. Due to Proposition \ref{ananas} we conclude that $(\mu_t)_{t\geq 0}$  are the marginal distributions  of the oscillatory steady state of the time inhomogeneous Markov chain with $T_0$-periodic jump  rates \eqref{uffona}.


We apply Lemma \ref{omm} considering the sequence $P_n:=\bb P_x \circ (\mu^{(n)},Q^{(n)})^{-1}$ and $\tilde{P}_n^{(\mu,Q)}:=\tilde{\bb P}_x \circ (\mu^{(n)},Q^{(n)})^{-1}$. The convergence $\tilde{P}_n^{(\mu,Q)}\to \delta_{(\mu,Q)}$
follows by Lemma \ref{lln} and the above observation that $\mu_t$ is the marginal of the oscillatory steady state of $\tilde{\mathbb P}$.

We \blu{now} observe that
\begin{equation}\label{euaz}
\frac 1n  \Ent\big(\tilde{P}_n^{(\mu,Q)} \big| P_n\big)\leq \frac 1n \Ent
\rosso{\Big(\tilde{\mathbb P}_x|_{[0,nT_0]}\Big| {\mathbb P}_x|_{[0,nT_0]}\Big)}\,.
\end{equation}
This is a special case of a general result that says that relative entropy is
decreasing under push forward. This follows directly by the variational representation of the \rosso{relative  entropy
(see e.g. \cite[Sec.~8, Appendix~1]{KL}).}
By a direct computation, using \eqref{RN-new}, we have that the right
hand side of \eqref{euaz} is given by
\begin{equation}\label{eus}
\tilde{\mathbb E}_x\left[\mu^{(n)}\left(r-\tilde r\right)+Q^{(n)}\left(\log\frac{\tilde r}{r}\right)\right]\,.
\end{equation}
\blu{Due to the definition of $\L_0$ and by Assumption (A3), the functions $r-\tilde r$ and $\log(\tilde r/r)$ are bounded in modulus. Hence, by the $L^1(\bbP_x)$--convergence in \eqref{mazurka1} and \eqref{mazurka2} in  Proposition \ref{lln}, we get that,} in the limit $n\to +\infty$, \eqref{eus}
converges to
\begin{equation*}
\int_0^{T_0}I_t(\mu_t,Q_t)dt=I(\mu,Q)=J(\mu,Q)\,.
\end{equation*}
\blu{This completes the proof of \eqref{entb}.}

\smallskip
It remains to prove that $(\sce J)=I$. Since $I$ is lower semicontinuous and $I\leq J$, then
by definition we have $(\sce J)\geq I$. We need to prove the converse inequality. Consider
\blu{$(\mu,Q)\in \Lambda$.}
 We construct a sequence $(\mu_n,Q_n)\in \Lambda_0$
such that $(\mu_n,Q_n)\to (\mu,Q)$ and moreover
$\varlimsup_{n\to +\infty}J(\mu_n,Q_n)\leq I(\mu,Q)$. This implies \blu{$(\sce J)(\mu,Q) \leq I(\mu,Q)$} and allows to conclude the proof.
limsup
\smallskip

\blu{We construct the above sequence by a diagonal procedure. To this aim
we let  $\Lambda_1 \subset \L$ be the collection of
elements $(\mu,Q)\in \Lambda$ such that there exists  $\epsilon>0$  for which  $\mu_t(x)\geq \epsilon$ and $ Q_t(y,z)\geq \epsilon$ for  all $t\in \cS_{T_0}$, $x\in V$ and $(y,z)\in E$.  Below we prove the following claim:
\begin{claim}\label{quadrifoglio}
 The following holds:
\begin{itemize}
\item[(i)] For any $(\mu,Q)\in \L$, there exists a sequence  $(\mu_n,Q_n)\in \Lambda_1$
such that $(\mu_n,Q_n)\to (\mu,Q)$ and moreover
$\varlimsup_{n\to +\infty}I(\mu_n,Q_n)\leq I(\mu,Q)$.
\item[(ii)] For any $(\mu,Q)\in \L_1$, there exists a sequence  $(\mu_n,Q_n)\in \Lambda_0$
such that $(\mu_n,Q_n)\to (\mu,Q)$ and moreover
$ \varlimsup_{n\to +\infty}I(\mu_n,Q_n) \leq I(\mu,Q)$.
\end{itemize}
\end{claim}
The above claim allows to conclude as follows. Let us write $d(\cdot, \cdot)$ for a metric on $\cM_*$ leading to the weak topology on $\cM_*$ (see \eqref{ho_fame}).
Fixed $(\mu,Q)\in \L$, by Item (i) of the above claim, we can find $(\mu_n,Q_n)$ in $\L_1$ with
$d\left( (\mu_n, Q_n), (\mu,Q) \right) \leq 1/n$ and $I(\mu_n,Q_n) \leq I(\mu,Q)+ n^{-1}$. By Item (ii) we can find
$( \mu^*_n,Q^*_n)$ in $\L_0$ with
$d\left( (\mu^*_n, Q^*_n), (\mu_n,Q_n ) \right) \leq 1/n$ and $I(\mu^*_n,Q^*_n) \leq I(\mu_n,Q_n)+ n^{-1}$. Then
$( \mu^*_n,Q^*_n) \to (\mu,Q)$ and $ \varlimsup_{n\to +\infty}I(\mu^*_n,Q^*_n) \leq I(\mu,Q)$.  Using that $I(\mu^*_n,Q^*_n) =J(\mu^*_n,Q^*_n) $ by definition of $J$, we conclude that  $ \varlimsup_{n\to +\infty}J(\mu^*_n,Q^*_n) \leq I(\mu,Q)$.
}

\medskip
\blu{The rest of this section is devoted to the proof of Claim \ref{quadrifoglio}.}

\subsection{\blu{Proof of Item (i) in Claim \ref{quadrifoglio}}}
\blu{Let $(\mu,Q)\in \L$.}
The sequence $(\mu_n,Q_n)$ is defined as
$$
(\mu_n,Q_n):=\frac 1n (\pi,Q^\pi)+\left(1-\frac 1n\right)(\mu,Q)\,.
$$
We point out that $\pi_t(y)$ can be estimated from below by the probability that $\xi_0=y$ and  that the Markov chain does not jump in the time interval $[0,t]$. Hence
\[ \pi_t (y) \geq \pi_0 (y) \exp \{ -\int _0^t r(y;s) ds \}\,.\]
Due to Assumption (A3) (cf. \eqref{1maggio}) and since $\pi_0$ is a positive measure, we conclude that $\min _y \inf_{t \in  [0,T_0]} \pi_t (y) >0$. As a byproduct of this bound and again \eqref{1maggio} we also conclude that  $Q^\pi_t (y,z)= \pi_t (y) r(y,z ;t)$  is bounded from below by a positive constant uniformly in $(y,z)\in E$ and $t \in [0,T_0]$. These observations imply that $( \pi, Q^\pi) \in \L_0$ and therefore that  $(\mu_n,Q_n)\in \L_0$. \blu{Trivially, $(\mu_n,Q_n)\to (\mu,Q)$ in $\cM_*$.}

Since $I$ is convex (cf. Remark \ref{I_convessa}) and $I(\pi,Q^\pi)=0$ we have
$$
\blu{I(\mu_n,Q_n)\leq \left(1-\frac 1n\right)I(\mu,Q)\,,}
$$
\blu{which implies that $\varlimsup_{n\to +\infty}I(\mu_n,Q_n) \leq I(\mu,Q)$.}
%
\subsection{\blu{Proof of Item (i) in Claim \ref{quadrifoglio}}}

\blu{Let $(\mu,Q)\in \L_1$ and let $\e>0$ be
such that   $\mu_t(x)\geq \epsilon$ and $ Q_t(y,z)\geq \epsilon$ for  all $t\in \cS_{T_0}$, $x\in V$ and $(y,z)\in E$.
 To built $(\mu_n,Q_n)$ we fix a sequence of nonnegative $C^\infty$--mollifiers  $\varphi_n$   with  support in $[-1/n,1/n]$ (see \cite{Brezis}).
We write $(\mu_n,Q_n)$ for the element in $\cM_*$ (cf. \eqref{ho_fame}) such that $\mu_n = [\mu_n]_t dt $, $Q_n=[Q_n]_t dt $ and
\begin{align}
&  [\mu_n]_t  (y) := \int _{\bbR} \mu_s(y) \varphi_n(t-s)ds\,,\label{defilare1}\\
& [Q_n]_t(y,z):= \int _{\bbR} Q_s(y,z) \varphi_n(t-s)ds\,.\label{defilare2}
\end{align}
Since the maps $t\mapsto \mu_t(y)$ and $t\mapsto Q_t(y,z)$ are in $L^1(dt):=L^1(\cS_{T_0},dt)$, we have that the mollified
  maps $t\mapsto [\mu_n]_t(y)$ and $t\mapsto [Q_n]_t(y,z)$ are $C^\infty$   and moreover converge
 in $L^1( dt)$, as $n\to \infty$, to  $t\mapsto \mu_t(y)$ and $t\mapsto Q_t(y,z)$  respectively (see e.g. \cite[Chp. 4]{Brezis}).   }

\blu{Let us first prove that $(\mu_n,Q_n) \in \L_0$.   It is simple to check that  $(\mu_n,Q_n) \in \L_1$ since $(\mu,Q)\in \L_1$. Note in particular that they solve the continuity equation and that  $[\mu_n]_t(y) \geq \e$ and $[Q_n]_t(y,z) \geq \e$. On the other hand, as already observed,    $[\mu_n]_t$ and $[Q_n]_t$ depend smoothly  on $t$, hence they are bounded from above. This conclude the proof that   $(\mu_n,Q_n) \in \L_0$}.

\blu{It remains to prove that $\varlimsup_{n\to \infty} I(\mu_n,Q_n ) \leq I(\mu,Q)$.
Since $(\mu, Q)\in \L_1$ and due to Assumption (A3),  we have  that the following $t$--functions
\begin{equation}\label{fuoco}Q_t(y,z) \,,  \qquad   Q_t(y,z) \log \mu_t(y) \,,\qquad   Q_t(y,z) \log r_t(y,z) \,, \qquad \mu_t(y) r(y,z;t)\end{equation}
belong to $L^1(dt)$.  Hence, $I(\mu,Q)$ can be written as  the sum among $(y,z) \in E$ of the following $(y,z)$--parameterized expressions (which are  meaningful since all terms below, with exception of at most one, are finite):
\begin{equation}\label{fuoco1001}
\begin{split}
& \int_{\cS_{T_0}}Q_t(y,z) \log Q_t(y,z)  dt   -\int_{\cS_{T_0}} Q_t(y,z) \log \mu_t(y)  dt \\
&-\int_{\cS_{T_0}} Q_t(y,z) \log r(y,z;t)  dt   -\int_{\cS_{T_0}}Q_t(y,z)   dt +\int_{\cS_{T_0}}\mu_t(y) r(y,z;t) dt\,.
\end{split}
\end{equation}
Since the map $(0,+\infty) \ni u \mapsto u \log u \in \bbR$ is convex and since the mollification is an average, we have
\[
[ Q_n]_t(y,z) \log [Q_n]_t(y,z)  \leq \int _{\bbR} \varphi_n(s)  Q_{t-s}(y,z) \log Q_{t-s}(y,z)ds \,.
\]
Hence,
\begin{equation}\label{abc1}
\begin{split}
 \int_{\cS_{T_0}} [ Q_n]_t(y,z) \log [Q_n]_t(y,z)  dt & \leq \int_{\cS_{T_0}}  dt \int _{\bbR} \varphi_n(s)  Q_{t-s}(y,z) \log Q_{t-s}(y,z)ds\\
 & = \int_{\cS_{T_0}}   Q_t(y,z) \log Q_t(y,z) dt\,.
\end{split}
\end{equation}
On the other hand, due to Assumption (A3) and the properties of mollifiers  stated after  \eqref{defilare2}, we have the  following limits in  $L^1(dt)$:
\begin{align}
& | [ Q_n]_t(y,z)-Q_t(y,z) | \to 0 \,,\label{abc2}\\
& | [ Q_n]_t(y,z)-Q_t(y,z) |  |\log r(y,z;t)|   \to 0 \,,\label{abc3}\\
& | [ \mu_n]_t(y)-\mu_t(y) |  r(y,z;t) \to 0\label{abc4} \,.
\end{align}
Finally,  we estimate
\begin{equation}\label{olivia100}
\begin{split}
& \big{|} [ Q_n]_t(y,z) \log [ \mu_n]_t(y)- Q_t(y,z)\log \mu_t(y)\big{|}
\\
& \leq \big{|} [ Q_n]_t(y,z) - Q_t(y,z)\big{ |}  \cdot \big|\log [ \mu_n]_t(y)\big |+  Q_t(y,z)\big | \log  \mu_t(y) -  \log  [ \mu_n]_t(y)  \big |\\
& \leq  |\log \e|\cdot \big{|} [ Q_n]_t(y,z) - Q_t(y,z)\big{ |}  +  Q_t(y,z)\big | \log \mu_t(y)-  \log [ \mu_n]_t(y) \big |\,.
\end{split}
\end{equation}
We already know that the first term in the r.h.s. goes to zero in $L^1(dt)$ (cf. \eqref{abc2}). On the other hand we can bound
\begin{equation}\label{cocomero}
| Q_t(y,z)\big | \log \mu_t(y)-  \log [ \mu_n]_t(y) \big | \leq   2 Q_t(y,z) |\log \e| \in L_1(dt)
\end{equation}
Since $t\mapsto [ \mu_n]_t(y)$ converges to $t\mapsto  \mu_t(y)$ in $L^1(dt)$, at cost to extract a subsequence we can suppose that the convergence is also Lebesgue almost everywhere. As a byproduct with \eqref{cocomero}, by dominated convergence, we conclude that also the second term  in the r.h.s. of \eqref{olivia100} goes to zero in $L_1(dt)$, thus implying the limit
\begin{equation}\label{abc5}
\big{|} [ Q_n]_t(y,z) \log [ \mu_n]_t(y)- Q_t(y,z)\log \mu_t(y)\big{|}\to 0 \text{ in $L_1(dt)$}\,.
\end{equation}
To conclude we write
$I(Q_n,\mu_n)$ as  the sum among $(y,z)\in E$ of
\begin{equation}\label{fuoco1000}
\begin{split}
& \int_{\cS_{T_0}} [ Q_n]_t(y,z) \log [Q_n]_t(y,z)  dt   -\int_{\cS_{T_0}} [ Q_n]_t(y,z) \log [\mu_n]_t(y)  dt \\
&-\int_{\cS_{T_0}} [ Q_n]_t(y,z) \log r(y,z;t) dt   -\int_{\cS_{T_0}} [ Q_n]_t(y,z)   dt +\int_{\cS_{T_0}} [\mu_n]_t(y) r(y,z;t) dt\,.
\end{split}
\end{equation}
Note that all the above integrals are finite since $(\mu_n,Q_n) \in \L_0$ and due to Assumption (A3). By \eqref{abc1} the limsup of the first addendum in \eqref{fuoco1000} is bounded from above by the first addendum in \eqref{fuoco1001}, while by
 using respectively  \eqref{abc5}, \eqref{abc3}, \eqref{abc2}, \eqref{abc4}, we get that   the limits of the other addenda in \eqref{fuoco1000}  are given by the similar addenda in \eqref{fuoco1001}. This conclude the proof that $\varlimsup_{n \to \infty} I(\mu_n,Q_n) \leq I(\mu,Q)$.}


\section{Proof of Theorem \ref{teo3}}\label{asia}
Recall  the continuous map $\cJ:\cM_+ ( E \times \cS_{T_0}) \to \cM _a ( E_s \times \cS_{T_0}) $
defined as $ \mathcal J(Q)(y,z,A):=Q(y,z,A)-Q(z,y,A)$, with the convention that $Q(y',z',A)=0$ if $(y',z')\not \in E$.
Due to the discussion preceding Definition \ref{deficit} it  only remains to show that the function
\[ \hat I(\mu,J):= \inf_{Q: \mathcal J(Q)=J}I(\mu,Q)\]
\blu{is convex and} equals the r.h.s. of \eqref{piscina}, \blu{and to derive \eqref{cuore}.}

\blu{ Convexity follows from  the convexity of $I$ and the affinity  of $\cJ$. Let us prove that $\hat I(\mu,J)$ equals the r.h.s. of \eqref{piscina}.}
Trivially, if $J=\cJ (Q) $ with $(\mu,Q)\in \L$, then $(\mu,J)\in \L_a$. \blu{On the other hand,  if $(\mu,J)\in \L_a$ then
$J=\cJ (Q) $ where $Q_t(y,z):=\max\{ J_t(y,z), 0\}$, in particular $(\mu,Q)\in \L$}.
 Since $I\equiv +\infty$ on $\L^c$, we conclude that
$\hat I(\mu,J)=+\infty$ if $(\mu,J) \not \in \L_a$, in agreement with the r.h.s. of \eqref{piscina}. Hence, from now on we  restrict to $(\mu,J) \in \L_a$.

\medskip

Given a current $J\in  \cM _a ( E_s \times \cS_{T_0}) $ we can write it uniquely in its Jordan decomposition $J = J^+ - J^-$. We recall that  $J^\pm$  are nonnegative  measures in $\cM_+ ( E_s \times \cS_{T_0}) $  with  disjoint supports. The antisymmetry of $J$ implies that
\[J^+(y,z,A) = J^-(z,y,A) \qquad \forall A\subset \cS_{T_0} \text{ measurable}\,.
\]
\blu{Since we restrict to $(\mu,J)  \in \L_a$, we have $J^+= J^+_t  dt$ and $J^-=J^-_t dt $, where
$J^+_t (y,z):= \max\{ J^+_t (y,z),0\}$ and $J^-_t (y,z):=- \min\{ J^+_t (y,z),0\}$.}
 Note that by property  (v) in Definition \ref{deficit} of $\L_a$, $J^+$ \blu{and $J^-$ have} support included in $E\times \cS_{T_0}$.

   All the flows $Q\in\cM _+ ( E \times \cS_{T_0}) $   such that  $\cJ(Q) = J$ can be characterized by the decomposition  $Q = J^+ +S$, where $S$ is an arbitrary  element of $\cM_+(E\times \cS_{T_0} )$ such that
   \[
\begin{cases}
 S( y,z,A) = S(z,y,A) & \text{ if }\;(y,z)\in E \text{ and } (z,y) \in E\,,\\
 S(y,z,A)=0 & \text{ if }\; (y,z)\in E \text{ and }  (z,y)\not \in E\,.
 \end{cases}
 \]

\begin{definition}\label{renzino}
We denote by  $\cS = \cS(\mu)$ the space of measures $S\in\cM _+ ( E \times \cS_{T_0}) $ such that  $S=S_t dt $, $S_t \in \bbR_+ ^E$,
\[
\begin{cases}
 S_t( y,z) = S_t(z,y) & \text{ if }\;(y,z)\in E \text{ and } (z,y) \in E\,,\\
 S_t(y,z)=0 & \text{ if }\; (y,z)\in E \text{ and }  (z,y)\not \in E\,.
 \end{cases}
 \]
 and, given $(y,z) \in E$, if $\mu_t (y)=0$ then $S_t(y,z)=0$ for a.e. \rosso{$t\in \cS_{T_0}$}.
\end{definition}
Recall that we restrict to $(\mu,J) \in \L_a$.  By the previous observations,
the flows $Q$ such that $(\mu,Q) \in \L$ and $\cJ(Q) = J$ are  characterized by the decomposition  $Q = J^+ +S$, where $S \in \cS$.

 Due to the previous observations we have
\begin{equation}\label{hosonno}
\hat I(\mu,J)= \inf_{S\in \cS }I(\mu,J^++S)
=  \inf_{S\in \cS}  \sum_{(y,z)\in E}  \int_0^{T_0}   \Phi \big( J^+_t(y,z)+S_t(y,z),  \mu_t(y) r_t(y,z)\big) dt  \,,
\end{equation}
where the infimum is among the symmetric elements $S$ as above. Note that  we have set $r_t(y,z):=r(y,z;t)$.
To solve the variational problem
\eqref{hosonno} it is enough to minimize for each $t$ and for each $(y,z)\in E$ the contribution in the r.h.s. of \eqref{hosonno} of  the \rosso{terms} associated to $(y,z)$ and to $(z,y)$ (if $(z,y)\in E$, otherwise one restricts only  to the \rosso{term} associated to $(y,z)$).

To this aim,  given $(v,w)\in E$ we set
\begin{equation}\label{bakeoff}
Q^{J,\mu}_t(v,w):=\frac{J_t(v,w)+\sqrt{J_t^2(v,w)+4\mu_t(v)\mu_t(w)r_t(v,w)r_t(w,v)}}{2}\,.
\end{equation}

\medskip

{\bf Case 1}. For  $ (y,z)\in E $ with $  (z,y)\not \in E$ we know that $S_t(y,z)=0$ \blu{and} $ J^+_t(y,z)=J_t(y,z)$ \blu{(see Definition \ref{deficit}-(v))}. Therefore, for all
$ t\in [0,T_0]$, we have
\begin{equation*}
\Phi \big( J^+_t(y,z)+S_t(y,z),  \mu_t(y) r_t(y,z)\big) = \Phi \big( Q^{J,\mu}_t(y,z)  ,  \mu_t(y) r_t(y,z)\big)\,.
\end{equation*}

{\bf Case 2}. Let us now take $(y,z)\in E $ such that $ (z,y) \in E$.  It is enough to minimize, for each $t\in [0,T_0]$, the contribution
\begin{equation}\label{rubik}
\Phi \big( J^+_t(y,z)+S_t(y,z),  \mu_t(y) r_t(y,z)\big)+\Phi \big( J^+_t(z,y)+S_t(z,y),  \mu_t(z) r_t(z,y)\big)\,,
\end{equation}
when varying the parameter $S_t(y,z)=S_t(z,y)$ in $\bbR_+$.
We define
\[ s_t:= S_t(y,z)= S_t(z,y)\,,\qquad j_t^+:= J^+_t (y,z)\,,\qquad j_t^-:= J_t^-(y,z)=\rosso{J^+(z,y)}\,.
\]

{\bf Subcase 2.a}. Supposing $ \mu_t(y)>0$ and $\mu_t(z)>0$,  by definition of $\Phi$ we have to minimise
(cf. \eqref{rubik})
\begin{equation}\label{lavoro}
\begin{split}
 \inf_{s_t\in \bbR_+ }\Big\{ \left(j^+_t+s_t\right)\log\frac
{\left(j^+_t+s_t\right)}{\mu_t(y)r_t(y,z)}& +\left(j^-_t+s_t\right)\log\frac
{\left(j^-_t+s_t\right)}{\mu_t(z)r_t(z,y)} \\
&+\mu_t(y)r_t(y,z)+\mu_t(z)r_t(z,y)-j^+_t-j^-_t -2 s_t\Big\}\,.
\end{split}
\end{equation}
By simple computations one gets that the minimizer is given by
$$
s_t=\frac{-\left(j_t^+ +j^-_t \right)+\sqrt{ (j_t^+-j_t^-)^2+4\mu_t(y)\mu_t(z)r_t(y,z)r_t(z,y)}}{2}\,.
$$
We point out that $s_t >0 $ since $\min(j_t^+ ,j_t^{-})=0$. It then follows that the infimum in \eqref{lavoro} equals
\begin{equation}
\label{girasole}
\Phi \big( Q^{J,\mu}_t(y,z),  \mu_t(y) r_t(y,z)\big)+\Phi \big( Q^{J,\mu}_t(z,y),  \mu_t(z) r_t(z,y)\big) \,.\end{equation}

{\bf Subcase 2.b}. If   $ \mu_t(y)=0$ and   $\mu_t(z)>0$, then by Property (iv) in Definition \ref{deficit}  and by Definition \ref{renzino} of $\cS$ for a.e. $t$ we have $j_t^+=0=s_t$. In this case, for a.e. $t$ the contribution \eqref{rubik} equals
\begin{equation}
j^-_t \log\frac
{j^-_t}{\mu_t(z)r_t(z,y)} +\mu_t(z)r_t(z,y)-j^-_t\,, \end{equation}
which again equals
 \eqref{girasole}.

{\bf \rosso{Subcases} 2.c, 2.d}. If   $ \mu_t(y)>0$ and   $\mu_t(z)=0$, or  $ \mu_t(y)=0$ and   $\mu_t(z)=0$, one gets that $s_t=0$ and the contribution \eqref{rubik} equals \eqref{girasole} by the same arguments used in Subcase 2.b.

Collecting all the above cases from Case 1 to Case 2.d , we get that
\begin{equation}\label{solare}
\hat I(\mu,J)=\int_0^{T_0}I_t(\mu_t, Q^{J,\mu}_t)\,dt
\end{equation}
\blu{for any $(\mu,J) \in \L_a$. This concludes the proof of \eqref{piscina}.}

Finally, the derivation of \eqref{cuore}  from the above formula can be done as  in \cite{BFG2} (cf. Theorem 6.1 there) by
\rosso{adapting} the conclusion there. Let us give more comments. \blu{Take $(\mu, J) \in \L_a$}. As for \cite[Eq. (6.6)]{BFG2} we have
\begin{multline*}
  \Psi \left( J_t(y,z), J_t^\mu(y,z);a_t ^\mu(y,z) \right)= \\
\Phi \left(
Q^{J,\mu}_t(y,z) , \mu_t (y) r(y,z;t) \right)+ \Phi \left(
Q^{J,\mu}_t(z,y) , \mu_t (z) r(z,y;t) \right)
\end{multline*}
if both $(y,z)$ and $(z,y)$ belong to $E$.  Hence in this case we have
\begin{multline}\label{maurizio}
\frac{1}{2}\left\{   \Psi \left( J_t(y,z), J_t^\mu(y,z);a_t ^\mu(y,z) \right)+  \Psi \left( J_t(z,y), J_t^\mu(z,y);a_t ^\mu(z,y) \right)
\right\}=\\
\Phi \left(
Q^{J,\mu}_t(y,z) , \mu_t (y) r(y,z;t) \right)+ \Phi \left(
Q^{J,\mu}_t(z,y) , \mu_t (z) r(z,y;t) \right)\,.
\end{multline}

Let us now suppose that $(y,z) \in E$ and $(z,y )\not \in E$. Then it must be  $J_t(y,z)  = Q^{J,\mu}_t(y,z)  \geq  0$ and
$ J^\mu_t(y,z)= \mu(y) r_t (y,z)\geq 0 $. Since $a^\mu_t (y,z)=0$ we have
 \[  \Psi \left( J_t(y,z), J_t^\mu(y,z);a_t ^\mu(y,z) \right)=\Phi \left( Q^{J,\mu}_t(y,z) , \mu_t (y) r(y,z;t) \right)\,.\]
 On the other hand,  we have $a^\mu_t (z,y)=0$,  $J_t(z,y)=-J_t(y,z)\leq 0 $ and $ J^\mu_t(z,y)=-J^\mu_t(y,z)\leq 0 $. Hence, by definition of $\Psi$, we have
 \[  \Psi \left( J_t(z,y ), J_t^\mu(z,y);a_t ^\mu(z,y ) \right)= \Psi \left( J_t(y,z), J_t^\mu(y,z);a_t ^\mu(y,z) \right)\,.\]
 Since moreover  $\Phi \bigl(
Q^{J,\mu}_t(z,y) , \mu_t (z) r(z,y;t) \bigr)= \Phi(0,0)=0$, also in this case we have \eqref{maurizio}. By symmetry we conclude that \eqref{maurizio} holds for any $(y,z) \in E_s$.
As a byproduct of the above observation, \eqref{rfq} and  \eqref{solare}, we get \eqref{cuore}.

\medskip

We conclude by discussing goodness and convexity of $\hat I$. Goodness follows from the goodness of $I$ by application of the contraction principle. On the other hand, by \eqref{arancione}, $\hat I(\mu,Q)$ equals the infimum of the convex rate functional $I$ on a suitable affine subspace, thus implying that $\hat I$ itself is convex.

\section{Proof of Theorem \ref{teo_GC1}   }\label{pizza}
In what follows, as done before,  we use the convention $0 \log 0 :=0$.
\subsection{Proof of \eqref{uva1}}
 Since both  $(\mu,Q)$ and $(\theta \mu, \theta Q) $ belong to $\L$   we can write
\begin{equation*}
\begin{split}
 I(\theta \mu,\theta Q;r)- I(\mu,Q;r)
& = \sum_{y,z}  \int_{0}^{T_{0}}ds
 \left[\begin{array}{c}
-Q_{s}(y,z)\log\ \frac{Q_{s}(y,z)}{\mu_{s}(y)r(y,z;s)}\\
+Q_{T_{0}-s}(z,y)\log \frac{Q_{T_{0}-s}(z,y)}{\mu_{T_{0}-s}(y)r (y,z;s)}
\end{array}\right]
\\&
 +\sum_y  \int_{0}^{T_{0}}ds \left[
\begin{array}{c}
\sum_z Q_{s}(y,z) -\mu_{s}(y)r(y;s )\\
-\sum _z Q_{T_{0}-s}  (z,y)+\mu_{T_{0}-s}(y)r(y;s)
\end{array}\right]\\
&
 = \sum_{y,z}  \int_{0}^{T_{0}}ds
 \left[\begin{array}{c}
-Q_{s}(y,z)\log\ \frac{Q_{s}(y,z)}{\mu_{s}(y)r(y,z;s)}\\
+Q_{s}(z,y)\log \frac{Q_{s}(z,y)}{\mu_{s}(y)r (y,z;T_0-s)}
\end{array}\right]
\\&
 +\sum_y  \int_{0}^{T_{0}}ds \left[
\begin{array}{c}
\sum_z Q_{s}(y,z) -\mu_{s}(y)r(y;s )\\
-\sum _z Q_{s} (z,y)+\mu_{s}(y)r(y;T_0-s)
\end{array}\right]\,.
 \end{split}
\end{equation*}
Therefore we have
\begin{multline}\label{pigna1}
 I(\theta \mu,\theta Q;r)- I(\mu,Q;r)
 = \\
  \int_{0}^{T_{0}}ds \Big[\sum_{y,z}
Q_{s}(y,z)\log  \frac{\mu_{s}(y)r (y,z;s )}{\mu_{s}(z)r(z,y;T_{0}-s)}+\sum_{y}\mu_{s}(y)\left(-r(y;s)+r(y;T_{0}-s)\right)\Big]\,.
 \end{multline}
On the other hand we have
\[
\int_{0}^{T_{0}}ds\sum_{y,z}Q_{s}(y,z)\log \frac{\mu_{s}(y)}{\mu_{s}(z)}=\int_{0}^{T_{0}}ds\sum_{y}\log\bigl(\rosso{\mu_{s}(y)}\bigr) \sum_{z}\left(Q_{s}(y,z)-Q_{s}(z,y)\right)\,.
\]
Using now the
continuity equation $\partial_{s}\mu_{s}(y)+\sum_{z}\left[Q_{s}(y,z)-Q_{s}(z,y)\right]=0$,
we obtain
\begin{equation}\label{pigna2}\int_{0}^{T_{0}}ds\sum_{y,z}Q_{s}(y,z)\log \frac{\mu_{s}(y)}{\mu_{s}(z)}
=\rosso{-}\sum_y\int_{0}^{T_{0}}ds\log \bigl(\mu_{s}(y)\bigr) \partial_{s}\mu_{s}(y)=0\,.
\end{equation}
As a byproduct of  \eqref{pigna1}  and \eqref{pigna2} one gets \eqref{uva1}.

\subsection{Proof of \eqref{uva2}}
Since $(\mu,Q)\in \L$ we can write
\begin{multline*}
 I(\theta \mu,\theta Q;r^{\rm R}) - I(\mu,Q;r)
 = \sum_{y,z}  \int_{0}^{T_{0}}ds
 \left[\begin{array}{c}
\begin{array}{c}
-Q_{s}(y,z)\log  \frac{Q_{s}(y,z)}{\mu_{s}(y)r(y,z;s)} \end{array}\\
+Q_{T_{0}-s}(z,y)\log \frac{\rosso{Q_{T_0-s}}(z,y)}{\mu_{T_{0}-s}(y)r (y,z;T_{0}-s)}
\end{array}\right]\\
 +\sum _y \int_{0}^{T_{0}}ds
 \left[\begin{array}{c}
\sum_{z}Q_{s}(y,z)-\mu_{s}(y)r(y;s)\\
\begin{array}{c}
-\sum_{z}Q_{T_{0}-s}(z,y)+\mu_{T_{0}-s}(y)r(y;T_{0}-s)\end{array}
\end{array}\right]\,.
 \end{multline*}
 By a local  change of variable $T_0-s\mapsto s$   the last expression
 in the r.h.s. is zero, while the first expression can be simplified. This leads to
 \begin{equation*}
 I(\theta \mu,\theta Q;r^{\rm R})- I(\mu,Q;r)  =\sum_{y,z} \int_{0}^{T_{0}}ds\,Q_{s}(y,z)\log \frac{\mu_{s}(y)r(y,z;s)}{\mu_{s}(z)r (z,y;s)} \,.
 \end{equation*}
By \eqref{pigna2} we can write the above r.h.s. as $\sum_{y,z} \int_{0}^{T_{0}}ds\,Q_{s}(y,z)\log \frac{r(y,z;s)}{r (z,y;s)}= \blu{S_{\rm tot}(Q;r) }$.

\subsection{Proof of \eqref{uva3}} Since $(\mu,Q) \in \L$ we can write
\begin{equation*}
\begin{split}
 I(\theta \mu, \theta Q; r^{\rm DR})&-I(\mu,Q; r)
 \\ &
 = \sum _{y,z} \int_0^{T_0}ds
\left[\begin{array}{c}
-Q_{s}(y,z)\log \frac{Q_{s}(y,z)}{\mu_{s}(y)r (y,z;s )}\\
+Q_{T_{0}-s}(z,y)\log \frac{Q_{T_{0}-s}(z,y)}{\mu_{T_{0}-s}(y)w_{T_{0}-s}^{-1}(y)r (z,y;T_{0}-s)w_{T_{0}-s}(z)}
\end{array}\right]\\
&+ \sum_y \int_0 ^{T_0} ds
\left[\begin{array}{c}
\sum_{z}Q_{s}(y,z)-\mu_{s}(y)r(y;s)\\
-\sum_{z}Q_{T_{0}-s}(z,y) +\mu_{T_{0}-s}(y)r(y;T_{0}-s)
\end{array}\right]\\
& = \sum _{y,z} \int_0^{T_0}ds
\left[\begin{array}{c}
-Q_{s}(y,z)\log \frac{Q_{s}(y,z)}{\mu_{s}(y)r (y,z;s )}\\
+Q_{s}(z,y)\log \frac{Q_{s}(z,y)}{\mu_{s}(y)w_{s}^{-1}(y)r (z,y;s)w_{ s}(z)}
\end{array}\right]\\
&+ \sum_y \int_0 ^{T_0} ds
\left[\begin{array}{c}
\sum_{z}Q_{s}(y,z)-\mu_{s}(y)r(y;s)\\
-\sum_{z}Q_{s}(z,y) +\mu_{s}(y)r(y;s)
\end{array}\right]\\
&= \int_{0}^{T_{0}}ds \sum_{y,z}
Q_{s}(y,z)\log \frac{\mu_{s}(y)w_{s}^{-1}(y)}{\mu_{s}(z)w_{s}^{-1}(z)}\\
& = \int_{0}^{T_{0}}ds \sum_{y,z}
Q_{s}(y,z)\log \frac{w_{s}(z)}{w_{s}(y)}\,.
\end{split}
\end{equation*}
We point out that  the second identity follows from a local  chance of variable $s\mapsto T_0-s$, while the forth  identity follows from \eqref{pigna2}.

By using the continuity equation $\partial_{s}\mu_{s}(z)=\sum_{y}\left[Q_{s}(y,z)-Q_{s}(z,y)\right]$ and  integrating by parts, we conclude the proof of \eqref{uva3} by observing that
\begin{multline*}
 \int_{0}^{T_{0}}ds \sum_{y,z} Q_{s}(y,z)\log \frac{w_{s}(z)}{w_{s}(y)}  =\int_{0}^{T_{0}}ds\sum_{z}\log \bigl( w_{s}(z)\bigr)\sum_{y}\left(Q_{s}(y,z)-Q_{s}(z,y)\right)\\
  = \sum_{z}\int_{0}^{T_{0}}ds \log \bigl( w_{s}(z)\bigr) \partial_{s}\mu_{s}(z)=
  - \sum_{z}\int_{0}^{T_{0}}ds\mu_{s}(z)
  \partial_{s}  \log \bigl( w_{s}(z)\bigr)
  = S_{\rm ex} (\mu;r)\,.
\end{multline*}

\subsection{Proof of \eqref{rio1}, \eqref{rio2} and \eqref{rio3}}
These last three identities  \rosso{follow} by  minimizing \eqref{uva1}, \eqref{uva2}, \eqref{uva3},  respectively. One needs to  observe that the map
$(\mu, Q) \mapsto (\theta \mu, \theta Q )$ is a bijection \rosso{on} $\L$ and to use  the identities
$S_{\rm naive}(\theta \mu,\theta Q;r)=-S_{\rm naive}(\mu,Q;r)$,  $S_{\rm tot} (\theta Q; r^{\rm R})=- S_{\rm tot} (Q;r )$,
$S_{\rm ex} (\theta \mu; r^{\rm DR} )=-S_{\rm ex} (\mu;r) $.  For the last  identity we observe that the accompanying measure $w^{\rm DR}_s$ associated to the rates $r^{\rm DR}( \cdot, \cdot;s)$ equals $w_{T_0-s}$.

\bigskip

{\bf Acknowledgements}.
\rosso{We thank the  anonymous referees for their  careful reading and suggestions}.
 A. Faggionato and D. Gabrielli  thank   the Institute Henri Poincar\'e   for the kind hospitality and  the support during the trimester  ``Stochastic Dynamics Out of Equilibrium'',  in which  they have  worked on the manuscript.

\bigskip




\end{document}